\documentclass{amsart}%
\usepackage{amsfonts}%
\usepackage{amsmath}%
\usepackage{amssymb}%
\usepackage{graphicx, epic,eepic, hyphenat, epsf}

\input xy
\xyoption{all}

\newtheorem{theorem}{Theorem}
\theoremstyle{definition}

\newtheorem{corollary}{Corollary}
\newtheorem{definition}{Definition}

\newtheorem{lemma}{Lemma}
\newtheorem{proposition}{Proposition}
\newtheorem{remark}{Remark}
\numberwithin{equation}{section}
\newcommand{\brak}[1]{\langle #1\rangle}

\DeclareMathOperator{\id}{Id}
\DeclareMathOperator{\rk}{rk}
\def\co{\colon\thinspace} 
\def\mf{\mathfrak}

\newskip\stdskip                      % standard vertical space
\begin{document}

\title{Universal Khovanov-Rozansky $\mf{sl}(2)$ cohomology}
\author{Carmen Caprau}

\address{Department of Mathematics, California State University, Fresno}
\email{ccaprau@csufresno.edu}

\date{}

\begin{abstract}
We generalize the Khovanov-Rozansky cohomology theory for $n=2$ by using a homogeneous potential that depends on two parameters, to obtain the universal  Khovanov-Rozansky $\mf{sl}(2)$-link cohomology. This theory is equivalent to the universal $\mf{sl}(2)$-link cohomology using foams, after tensoring both theories with appropriate rings. 

\end{abstract}
\maketitle

\section{Introduction}
In \cite{KhR1} Khovanov and Rozansky $(KR)$ constructed for each $n >0$ a bigraded rational (co)homology theory categorifying the $\mf{sl}(n)$-link polynomial (the case $n=0$ is treated in \cite{KhR2}). Their construction uses matrix factorizations with potential $x^{n+1}$ associated to certain planar graphs, and for $n=2$, the corresponding homology is equivalent to the Khovanov homology defined in \cite{Kh1}. Gornik \cite{G} carried out a deformation of the $KR$-theory with potential $x^{n+1} - (n+1)\beta^n x$, for $\beta \in \mathbb{C},$ and Rasmussen \cite{Ras2} and Wu ~\cite{W} investigated $KR$-homologies given by a general non-homogeneous monic potential with degree $n+1$ and complex coefficients. $KR$-construction can be generalized to give the \textit{universal} matrix factorization link homologies for all $n >0,$ by working with a general homogeneous potential.  Recently, Mackaay and Vaz \cite{MV2} worked out this generalization for $n=3,$ and proved that the universal rational $\mf{sl}(3)$-matrix factorization link homology is equivalent to the foam link homology in \cite{MV1} tensored with $\mathbb{Q}.$ 

In \cite{CC2} the author constructed the universal $\mf{sl}(2)$-link cohomology via foams modulo local relations, in the spirit of \cite{BN} and \cite{Kh2} (see also \cite{CC1}). In this paper we introduce the universal rational Khovanov-Rozansky link cohomology for $n=2,$ and show that it is isomorphic to the foam $\mf{sl}(2)$-link cohomology in \cite{CC2}, after both theories are tensored with appropriate rings. To obtain the universal $\mf{sl}(2)$-matrix factorization theory, we consider a potential $p$ that depends on two parameters $a$ and $h,$ and that satisfies $ \partial p /\partial x = 3 (x^2 - hx - a).$

$KR$-construction starts from a certain version of the calculus developed by Murakami, Ohtsuki and Yamada \cite{MOY}, calculus which involves planar trivalent graphs. $KR$-graphs contain two types of edges, namely oriented edges and unoriented thick edges. For our purpose, we consider graphs that are obtained from the latter ones by ``erasing" all unoriented thick edges.

\section {Webs and matrix factorizations}\label{factorizations}

\subsection{The web space}

A \textit{web} with boundary $B$ is a planar graph with univalent and bivalent vertices. The univalent vertices correspond to boundary points, such as the boundary $\partial{T}$ of a tangle, and bivalent vertices have either indegree 2 or outdegree 2. Specifically, the two arcs incident with a bivalent vertex are either oriented ``in'' or ``out'' as shown below: 
\[ \raisebox{-5pt}{\includegraphics[height=.2in]{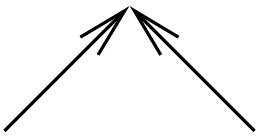}} \quad \text {or} \quad \raisebox{-5pt}{\includegraphics[height=.2in]{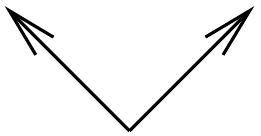}}\]
We call bivalent vertices \textit{singular points}. A \textit{closed web} is a web with empty boundary. We also allow webs with no bivalent vertices, thus oriented arcs or loops. We denote by $\textit{Foams}(B)$ the category whose objects are web diagrams with boundary $B,$ and whose morphisms are \textit{singular cobordisms}---called \textit{foams}---between such webs, regarded up to boundary-preserving isotopies. As morphisms, we read cobordisms from bottom to top, and we compose them by stacking one on top the other. 

Let $L$ be a link in $S^3.$ We fix a generic planar diagram $D$ of $L$ and resolve each crossing in two ways, as in Figure~\ref{fig:resolutions}. We refer to the diagram on the right as the \textit{oriented resolution}, and to the one on the left as the \textit{singular resolution}. We remark that the dotted lines should not be considered as edges---we prefer to draw them in order to record that there was a crossing before. 
\begin{figure}[ht]
$$\xymatrix@C=20mm@R=1.5mm{
  &  \includegraphics[height=0.5in]{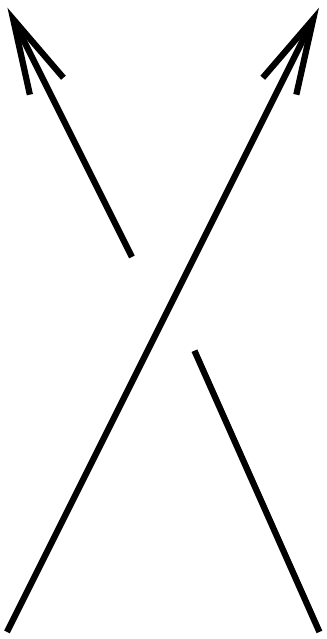} \ar[ld]_0\ar[rd]^1& \\
\includegraphics[height=0.5in]{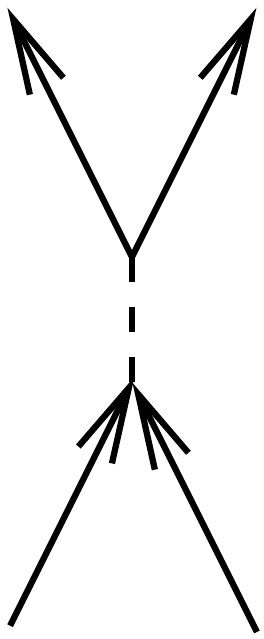} & & 
\includegraphics[height=0.5in]{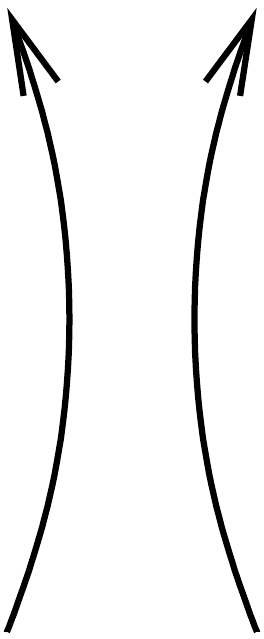} \\
  & \includegraphics[height=0.5in]{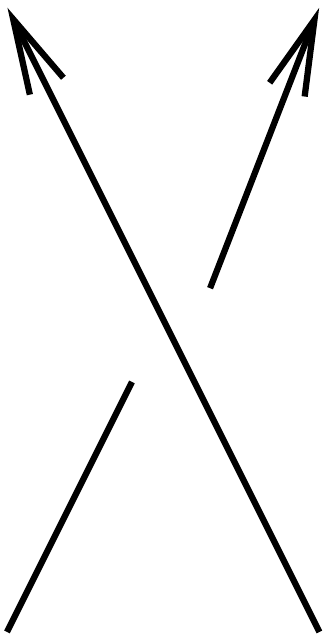} \ar[ul]^1 \ar[ur]_0&
}$$
\caption{The two ways of resolving crossings}
\label{fig:resolutions}
\end{figure}

As webs have singular points, foams have \textit{singular arcs} (or \textit{singular circles}) where orientations disagree. Basic cobordisms, as those between the two different resolutions of a crossing are depicted in Figure~\ref{fig:saddles}, where the arc colored red is a singular arc.

\begin{figure}[ht!]
\begin{center}
\includegraphics[height=.65in]{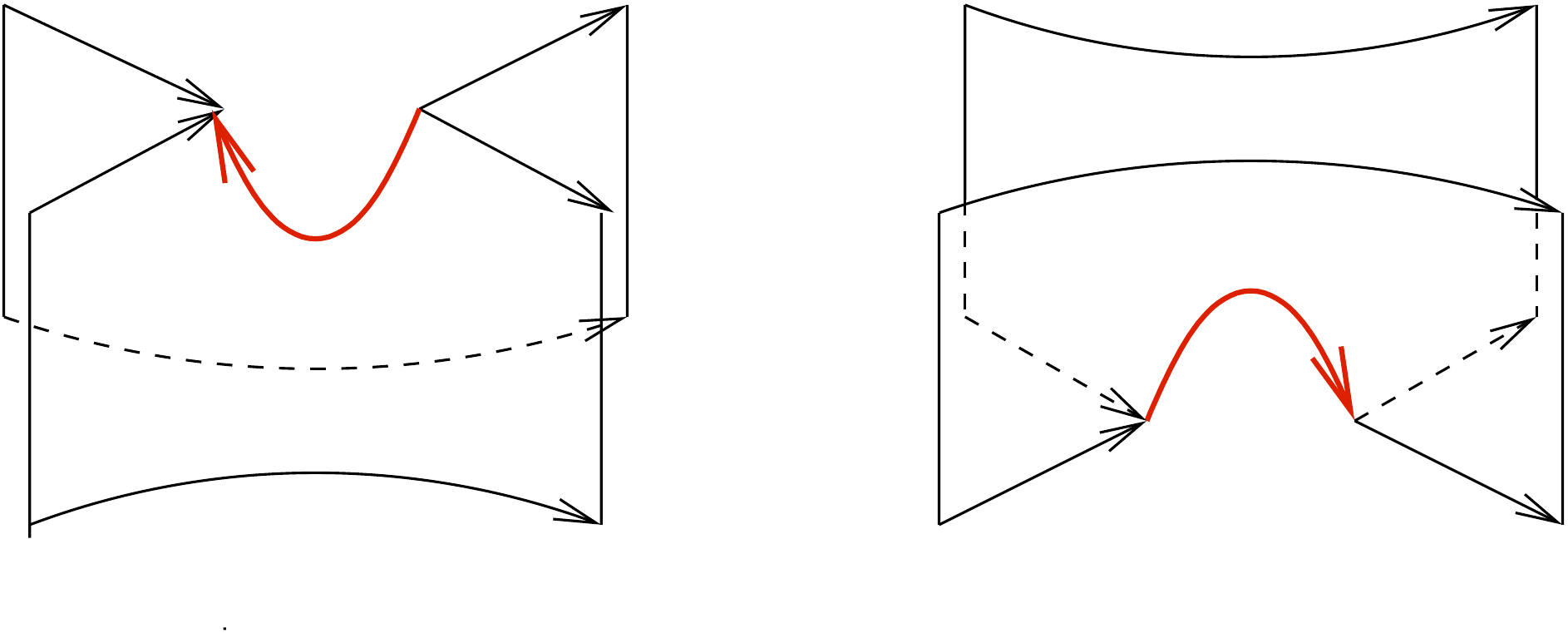}\end{center}
\caption{Singular saddles}
\label{fig:saddles}
\end{figure}

A diagram $\Gamma$ obtained by resolving all crossings of $D$---in one of the two possible ways explaned above---is a disjoint union of closed webs. There is a unique way to assign a Laurent polynomial $\brak{\Gamma} \in \mathbb {Z} [q, q^{-1}]$ to each closed web $\Gamma$ so that it  satisfies the skein relations explained in Figure~\ref{fig:skein relations}. (We remark that these are exactly the graph skein relations for $n= 2$ given in~\cite[Figure 3]{KhR1}, where all thick edges are erased.)

\begin{figure}[ht]
$$\xymatrix@R = 2mm{
\raisebox{-10pt}{\includegraphics[height=.35in]{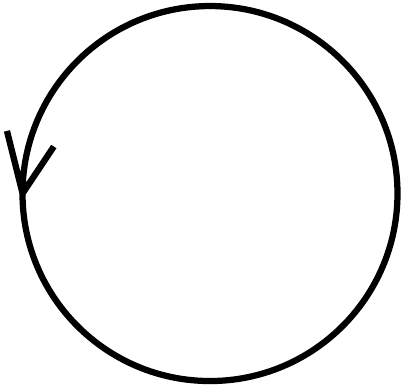}}\,\,\, = \,\,\, q + q^{-1}\hspace{1cm}
\raisebox{-30pt}{\includegraphics[height=.9in]{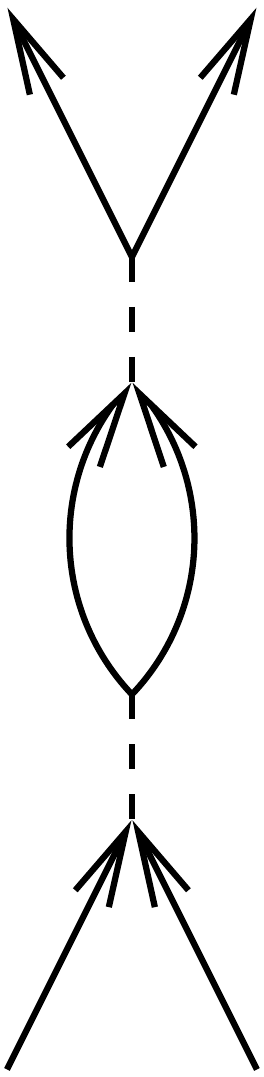}}\, = \,\,(q+ q^{-1})\,\raisebox{-30pt}{\includegraphics[height=.9in]{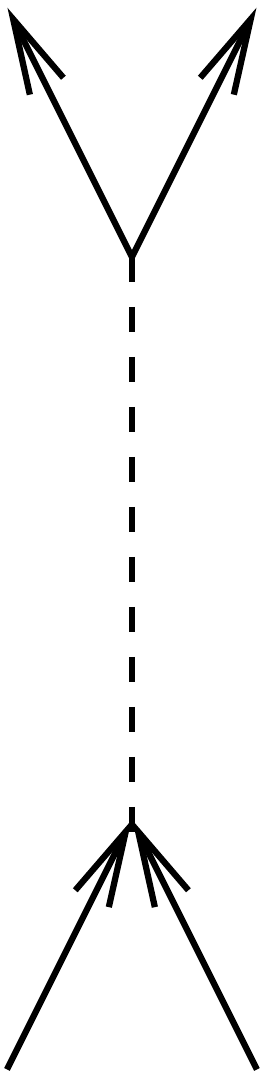}} \\ \raisebox{-15pt}{\includegraphics[height=.5in]{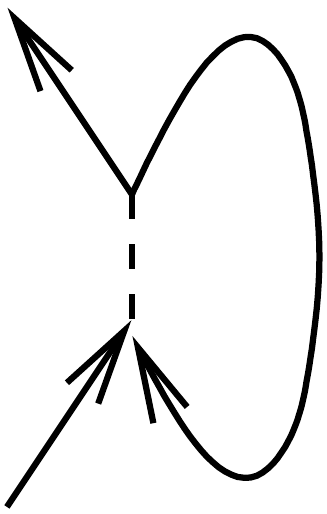}} = \raisebox{-15pt}{\includegraphics[height=.5in]{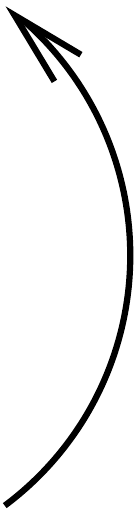}} \hspace{2cm}\raisebox{-15pt}{\includegraphics[height=.5in]{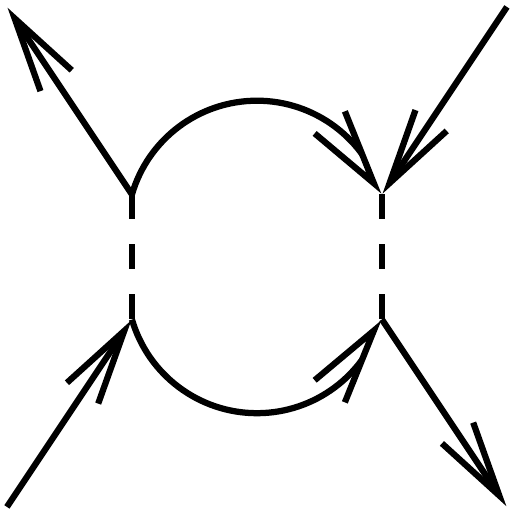}}\,\, =\,\, \raisebox{-15pt}{\includegraphics[height=.45in]{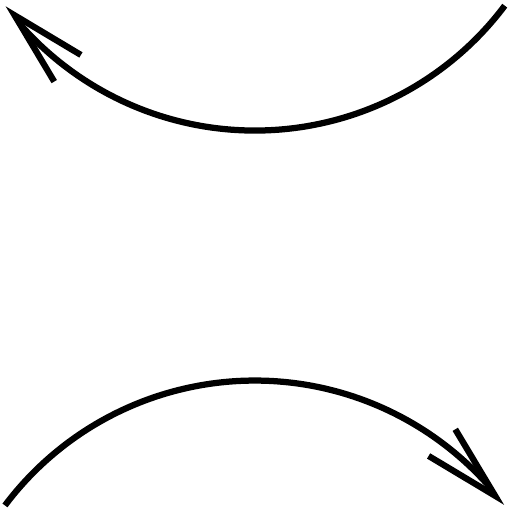}} \\
\raisebox{-30pt}{\includegraphics[height=.9in]{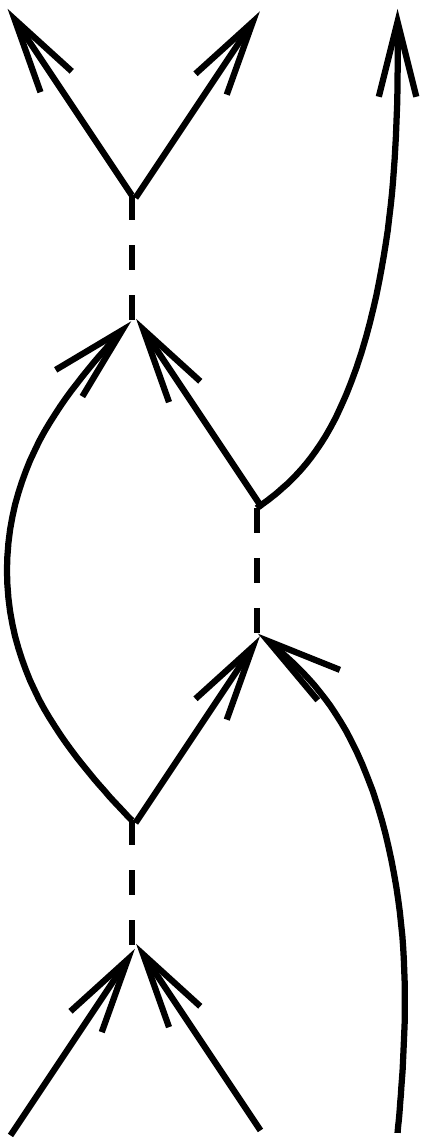}}\,\, +\,\, \raisebox{-30pt}{\includegraphics[height=.9in]{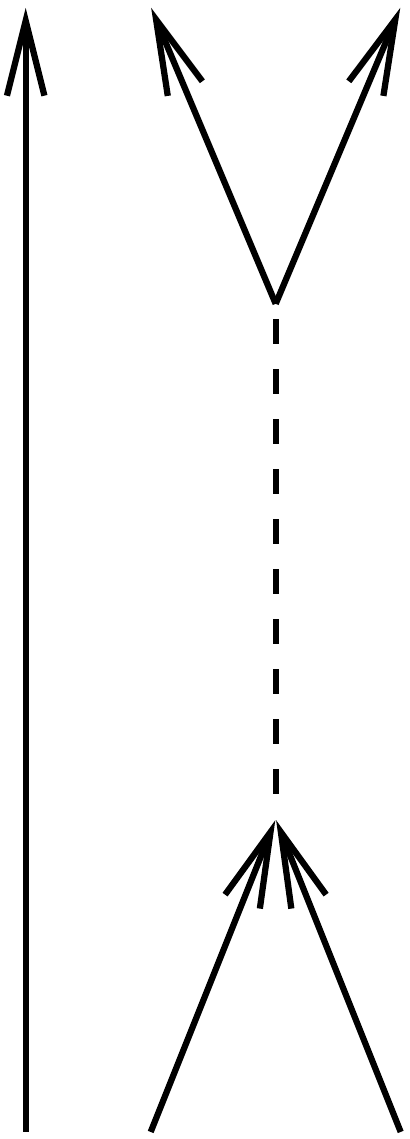}} \,\,=\,\, \raisebox{-30pt}{\includegraphics[height=.9in]{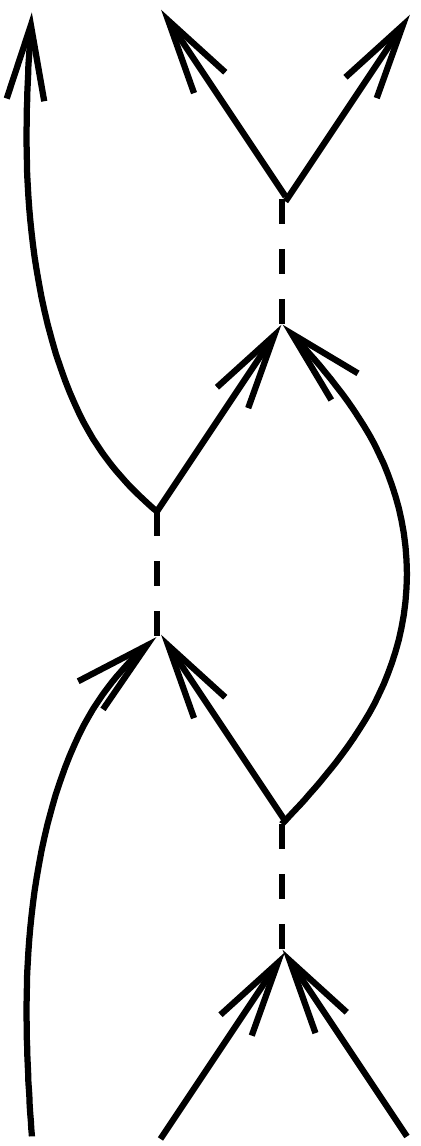}}\,\, + \,\,\raisebox{-30pt}{\includegraphics[height=.9in]{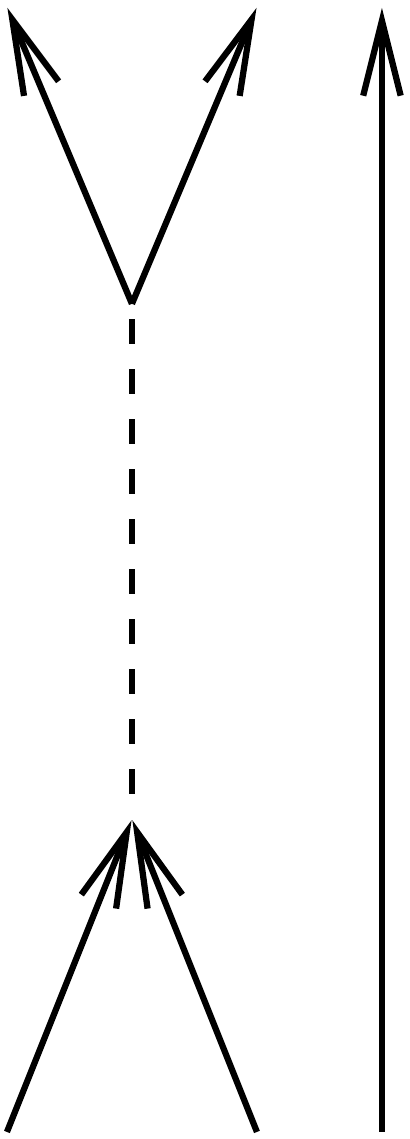}} 
}$$
\caption{Skein relations}
\label{fig:skein relations}
\end{figure}

The \textit{bracket} of $D$ is defined by $\brak{D} = \sum_{\Gamma} \pm q^{\alpha(\Gamma)}\brak{\Gamma},$ where the sum is over all resolutions of $D$ and $\alpha(\Gamma)$ is determined by the rules in Figure~\ref{fig:decomposition of crossings}.

\begin{figure}[ht]
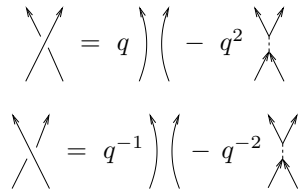

$$\xymatrix@R=2mm{
\raisebox{-12pt}{\includegraphics[height=0.4in]{poscros.pdf}}\,\, = \,\,q\,\,\raisebox{-12pt} {\includegraphics[height=0.4in]{orienresol.pdf}}\,\, -\,\, q^2\,\,\raisebox{-12pt} {\includegraphics[height=0.4in]{singresol.pdf}}\\
\raisebox{-12pt}{\includegraphics[height=0.4in]{negcros.pdf}} \,\,=  \,\,q^{-1}\,\raisebox{-12pt} {\includegraphics[height=0.4in]{orienresol.pdf}}\, -\,q^{-2} \,\raisebox{-12pt} {\includegraphics[height=0.4in]{singresol.pdf}}
}$$
\caption{Decomposition of crossings}
\label{fig:decomposition of crossings}
\end{figure}

It is well know that $\brak{D_1} = \brak{D_2}$ whenever $D_1$ and $D_2$ are related by a Reidemeister move, hence $\brak{L}:= \brak{D}$ is an invariant of the oriented link $L.$ Excluding the rightmost terms in Figure~\ref{fig:decomposition of crossings}, we obtain the skein relation $$q^2 \raisebox{-8pt}{\includegraphics[height=.32in]{negcros.pdf}} - q^{-2} \raisebox{-8pt}{\includegraphics[height=.32in]{poscros.pdf}} = (q - q^{-1}) \raisebox{-8pt}{\includegraphics[height=.32in]{orienresol.pdf}}$$ which yields the quantum $sl(2)$-polynomial of the link $L$ (thus the unnormalized Jones polynomial---its categorification is introduced in~\cite{Kh1}).

\subsection{From webs to matrix factorizations}

We mimic Khovanov's and Rozansky's work in ~\cite{KhR1} for $n =2$ case, by replacing the polynomial $p(x) = x^3$ with $p(a, h, x) = x^3 -\frac{3}{2} h x^2- 3ax,$ where $a$ and $h$ are formal variables. Note that $p(a, h, x)$ was chosen in such a way that $ \frac{1}{3} \partial p/ \partial x = x^2 -hx -a.$  We assume that the reader is somewhat familiar with~\cite{KhR1}, and we only briefly recall some of its content.

Given a graded ring $R$ (we consider only polynomial rings) and a homogeneous element $\omega \in R,$ the category $\textit{mf}_{\omega}$ of graded matrix factorizations with potential $\omega$ has objects  2-periodic chains
\[ M_0 \stackrel{d_0}{\rightarrow} M_1 \stackrel{d_1}{\rightarrow} M_0, \] 
where $M_0, M_1$ are graded free $R$-modules and $d_0, d_1$ are homomorphisms such that $d_0d_1 = \omega = d_1d_0.$ We call such an object an $(R, \omega)$-factorization. Morphisms in $\textit{mf}_{\omega}$ are degree-preserving maps of $R$-modules $M_0 \to N_0, M_1 \to N_1$ that commute with differentials. In this paper, the morphisms $d_0$ and $d_1$ have degree $3$ (thus $\deg(\omega) = 6$) and a homotopy has degree $-3.$ 

The category $\textit{hmf}_\omega$ of graded matrix factorizations up to chain homotopies has the same objects as $\textit{mf}_\omega$ but fewer morphisms, as homotopic morphisms in $\textit{mf}_\omega$ are declared the same in $\textit{hmf}_\omega.$ We denote by $\{r\}$ the grading shift up by $r,$ and by $\brak{s}$ the cohomological shift functor for matrix factorizations. Specifically, $M \brak{1}$ has the form $M_1 \stackrel {-d_1}{\rightarrow} M_0 \stackrel{-d_0}{\rightarrow} M_1.$ The cohomological shift functor for chain complexes is denoted by $[s].$ 

Given a pair of elements $a_1, b_1 \in R,$ we denote by $(a_1, b_1)$ the factorization with potential $a_1b_1$
\[R \stackrel{a_1}{\longrightarrow} R \{3 -\deg a \} \stackrel{b_1}{\longrightarrow} R, \]
where $a_1$ and  $b_1$ act on $R$ by multiplication. The middle $R$ was shifted so that the differentials would have degree $3$. Given a finite set of pairs $(a_i, b_i), 1 \leq i \leq n$, we denote their tensor product over $R$ by $(\textbf{a, b}) : = \otimes_i(a_i, b_i)$, where $\textbf{a} = (a_1, \cdots, a_n)$ and $ \textbf{b} = (b_1,\cdots,  b_n).$ Note that $(\textbf{a, b})$ is the factorization with potential $\omega = \sum_i a_i b_i,$ and we sometimes prefer to write it in the \textit{Koszul matrix} form:
\[ (\textbf{a, b}) = \left ( \begin{array}{cc} a_1 & b_1 \\ a_2 & b_2 \\ \cdots & \cdots \\ a_n & b_n \end{array} \right ). \]
Let $c \in R$ and denote by $[ij]_c$ the \textit{elementary row operation} which transforms
\[ \left ( \begin{array}{cc} a_i & b_i \\ a_j & b_j \end{array} \right ) \stackrel{[ij]_c}{\longrightarrow}  \left ( \begin{array}{cc} a_i+ca_j & b_i \\ a_j & b_j -cb_i\end{array} \right )\]
 and leaves the remaining rows of the Koszul matrix unchanged. An elementary row operation corresponds to a change of the basis vector of the free $R$-module underlying the corresponding factorization, thus it takes a Koszul factorization $(\textbf{a, b})$ to an isomorphic factorization in $\textit{hmf}_\omega.$

As noted by Rasmussen~\cite{Ras2}, if $b_i$ and $b_j$ are relatively prime, we can apply a \textit{twist} via the map $R \to R$ which sends $x \to kx,$ for some $k \in R,$ and obtain an isomorphism of factorizations
 \[ (a_i, b_i) \otimes (a_j, b_j) \cong (a_i + k b_j, b_i) \otimes (a_j - kb_i, b_j).\]

Using Koszul matrix form, the above isomorphism is written as
 \[ \left ( \begin{array}{cc} a_i & b_i \\ a_j & b_j \end{array} \right ) \stackrel{\cong}{\longrightarrow}  \left ( \begin{array}{cc} a_i + kb_j & b_i \\ a_j - kb_i & b_j \end{array} \right).\]
 
Starting with a link or tangle diagram $D,$ we put marks on each of its arcs. Resolving all of its crossings as explained in Figure~\ref{fig:resolutions}, we obtain web diagrams $\Gamma$ with marked arcs, and we denote by $m(\Gamma)$ the set of all marks of $\Gamma.$ We associate to each mark $i$ the polynomial $p(a, h, x_i) = x_i^3 -\frac{3}{2}h x_i^2 -3ax_i.$ When working with a tangle diagram, each of its resolutions has internal and external marks, where the later ones are the boundary points $i \in B$ of the tangle diagram. 

Let $R = \mathbb{Q}[a,h, x_i], i \in m(\Gamma),$ be the polynomial ring with rational coefficients and variables $a, h$ and $x_i$ (over all marks $i$), and let $R'$ be its subring $\mathbb{Q}[a,h, x_i], i \in B.$ We introduce a grading on $R$ and $R'$ by letting $\deg(x_i) = 2,  i \in m(\Gamma), \deg(a) = 4$ and $\deg(h) =2.$ To a web diagram $\Gamma$ we assign a graded factorization $\overline{C}(\Gamma)$ with (degree 6) homogeneous potential $\omega(\Gamma) = \sum_{i \in B} o(i)p(a, h, x_i),$ where $o(i) \in \{\pm1\} $ are ``orientations'' of boundary points, given by the orientation of $\Gamma$ at these points. 

To an oriented arc $l$ between two neighboring marks $i, j$ and oriented from $i$ to $j$ we assign the factorization $\overline{L}_i^j$ over the ring $R = \mathbb{Q}[a, h, x_i, x_j],$ and with potential $$\omega (\overline{L}_i^j)= p (a, h, x_j) -p (a, h, x_i) = x_j^3 -x_i^3 -\frac{3}{2} h (x_j^2 -x_i^2) -3a (x_j -x_i).$$
 \[\overline{L}_i^j \co \quad \mathbb{Q}[a,h,x_i,x_j]\stackrel{\overline{\pi}_{ij}}{\longrightarrow} \mathbb{Q}[a,h,x_i,x_j]\{-1\} \stackrel{x_j-x_i}{\longrightarrow}  \mathbb{Q}[a,h,x_i,x_j] \]

where $\overline{\pi}_{ij} =\displaystyle\frac {\omega(\overline{L}_i^j)}{x_j-x_i} = x_i^2 + x_ix_j + x_j^2 -\frac{3}{2}h(x_i + x_j) -3a$.

To an oriented circle with one mark $i$ we assign the factorization $\overline{L}_i^i,$ which is the quotient of $\overline{L}_i^j$ by the relation $x_j = x_i.$ We obtain a 2-periodic chain complex of $\mathbb{Q}[a,h,x_i]$-modules
  \[
\mathbb{Q}[a,h,x_i]\stackrel{\overline{\pi}_{ii}}{\longrightarrow} \mathbb{Q}[a,h,x_i]\{-1\} \stackrel{0}{\longrightarrow}  \mathbb{Q}[a,h,x_i], \] 
where $\overline{\pi}_{ii} = 3(x_i^2 - h x_i -a)$. This complex has cohomology only in degree 1, namely $$H^1(\overline{L}_i^i) = \mathbb{Q}[a,h,x_i]/(x_i^2 -hx_i - a) \{-1\}, \quad H^0(\overline{L}_i^i) = 0.$$ 

Let $\mathcal{A} : = \mathbb{Q}[a, h, X]/(X^2 - hX -a)$ and $\iota \co \mathbb{Q}[a, h] \to \mathcal{A}$ be the inclusion map $\iota(1) = 1.$ We identify $\mathbb{Q}[a, h,x_i]/(x_i^2 -hx_i- a)$ with $ \mathcal{A}$ by taking $x_i^k \in \mathbb{Q}[a, h, x_i]/(x_i^2 - hx_i -a)$ to $X^k \in \mathcal{A}.$ As a module over $\mathbb{Q}[a, h], \mathcal{A}$ is free with generators 1 and $X.$ We make $\mathcal{A}$ graded by giving to 1 degree $-1$ and to $X$ degree 1.

To an oriented circle without marks we associate the 2-periodic  chain complex of $\mathbb{Q}[a,h]$-modules $0 \rightarrow \mathcal{A} \rightarrow 0$ and denote it $\mathcal{A} \brak{1},$ following \cite{KhR1}. Note that $\mathcal{A}\brak{1} \cong \overline{L}_i^i$ as 2-periodic complexes of $\mathbb{Q}[a,h]$-modules, up to homotopies. The isomorphism takes $X^i \in \mathcal{A}$ to $x^i \in (\overline{L}_i^i)^1,$ for $0 \leq i \leq 1,$ and graphically it consists of adding a mark to a circle with no marks.

To diagrams $\Gamma^0$ and $\Gamma^1$ as in Figure~\ref{maps} we associate the factorizations $\overline{C}(\Gamma^0)$ and $\overline{C}(\Gamma^1)$ over the ring $R = \mathbb{Q}[a, h, x_1, x_2, x_3, x_4]$ (note that $R = R'$):
\begin{align*}
\overline{C}(\Gamma^0) &= (\overline{\pi}_{41}, x_1 -x_4) \otimes (\overline{\pi}_{32}, x_2 -x_3) \\
\overline{C}(\Gamma^1) &= (\overline{u}_{1}, x_1 + x_2 -x_3 -x_4) \otimes (\overline{u}_{2}, x_1x_2 -x_3x_4)
\end{align*}
 with $\omega(\overline{C}(\Gamma^0)) = p(a, h, x_1) + p(a, h, x_2) - p(a, h, x_3) - p(a, h, x_4) = \omega(\overline{C}(\Gamma^1))$ and 
 \begin{align*} \overline{u}_1 &= (x_1 + x_2)^2 + (x_1 + x_2)(x_3 + x_4) + (x_3 + x_4)^2 -3x_1x_2 -\frac{3}{2} h( x_1+x_2+x_3 + x_4) -3a \\ \overline{u}_2 &= -3(x_3+x_4) +3h.\end{align*} 
 Note that \begin{align*} \omega &= \overline{\pi}_{41}(x_1-x_4) + \overline{\pi}_{32}(x_2-x_3) \\ &= \overline{u}_1\cdot (x_1+x_2 -x_3 -x_4) + \overline{u}_2 \cdot (x_1x_2-x_3x_4).\end{align*} Precisely we have
 \[
\overline{C}(\Gamma^0)\co \quad
 \left( \begin{array}{c}
 R \\ R\{-2\}
 \end{array}\right) \stackrel{P_0}{\longrightarrow}
  \left( \begin{array}{c}
  R\{-1\} \\ R\{-1\}
  \end{array} \right) \stackrel{P_1}{\longrightarrow}
  \left(\begin{array}{c}
  R \\ R\{-2\} \end{array} \right)
 \]
 
 \[P_0 = 
 \left(\begin{array}{cc}
\overline{\pi}_{41}  & x_2-x_3 \\ 
 \overline{\pi}_{32}  & x_4-x_1\end{array}
 \right),  \quad P_1 = 
 \left(\begin{array}{cc}
 x_1 - x_4 & x_2-x_3 \\ 
\overline {\pi}_{32}  & -\overline{\pi}_{41} \end{array} \right)
 \]
 
 \[ \overline{C}(\Gamma^1)\co \quad
 \left( \begin{array}{c}
 R\{-1\} \\ R\{-1\}\end{array}
 \right) \stackrel{Q_0}{\longrightarrow}
 \left(\begin{array}{c}
 R\{-2\} \\ R\end{array}
 \right) \stackrel{Q_1}{\longrightarrow} \left(
 \begin{array}{c} R\{-1\} \\ R\{-1\} \end{array} \right)
 \]
 
 \[Q_0 = 
 \left( \begin{array}{cc}
\overline {u}_1  & x_1x_2 - x_3x_4 \\ 
\overline {u}_2 & x_3 +  x_4- x_1 - x_2 \end{array}
 \right), \quad Q_1 =
 \left(\begin{array}{cc}
 x_1 + x_2 - x_3 - x_4 & x_1 x_2 - x_3x_4 \\ 
 \overline{u}_2 & - \overline{u}_1 \end{array}\right)
 \]

\begin{figure}[ht]
\includegraphics[height=1in]{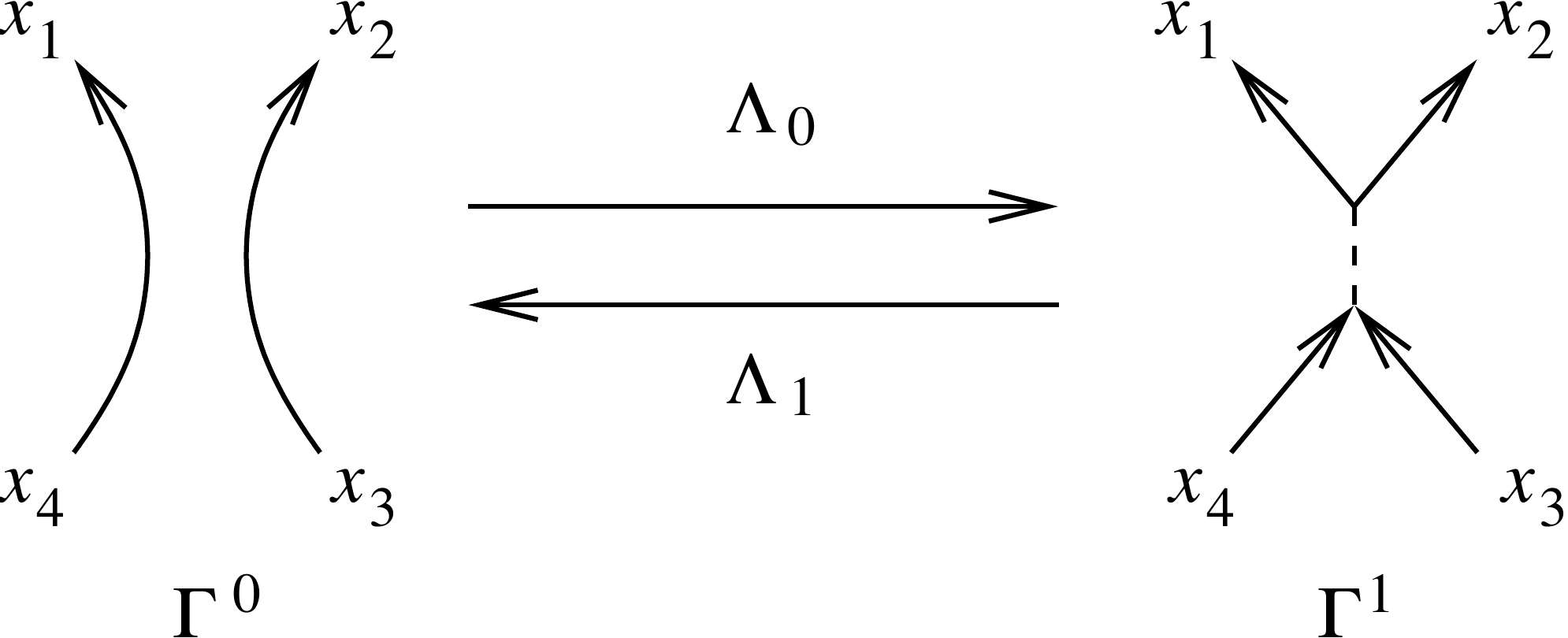} 
\caption{Webs $\Gamma^0$ and $\Gamma^1$}
\label{maps}
\end{figure}

Written in Koszul matrix form, $\overline{C}(\Gamma^0) = (\mathbf{a}, \mathbf{b})$ and $\overline{C}(\Gamma^1) = ( \mathbf{c}, \mathbf{d})\{-1\}$, where
\[(\mathbf{a}, \mathbf{b}) =
 \left(\begin{array}{cc}
\overline{\pi}_{41}  & x_1-x_4 \\ 
\overline{\pi}_{32}  & x_2-x_3\end{array} 
\right) \quad (\mathbf{c}, \mathbf{d}) = 
\left( \begin{array}{cc}
\overline{u}_1  & x_1+x_2-x_3-x_4 \\ \overline{u}_2 & x_1x_2-x_3x_4\end{array}\right),\]
and where the shift $\{-1\}$ in $\overline{C}(\Gamma^1)$ is applied to the second row 
$$R \stackrel {\overline{u}_2}{\longrightarrow} R\{1\} \stackrel{x_1x_2-x_3x_4}{\longrightarrow} R.$$
We shifted the degrees of $R$ so that each differential above has degree 3.

Finally, we define the $\overline{C}(\Gamma)$ as the tensor product of $\overline{C}(\Gamma^1)$ over all singular resolution, of $\overline{L}_i^j$ over all arcs $l,$ and of $\mathcal{A} \brak{1}$ over all loops with no mark. The tensor product is considered over appropriate rings, so that $\overline{C}(\Gamma)$ is a free module of finite rank over $R$, and we treat it as a graded factorization---with infinite rank---over the subring $R'.$ 

\begin{lemma}
Given any web diagram $\Gamma,$ its associated factorization $\overline{C}(\Gamma)$ lies in $\textit{hmf}_\omega.$ Moreover, if $\Gamma'$ is obtained from $\Gamma$ by placing a different collection of internal marks then there is a canonical isomorphism $\overline{C}(\Gamma') \cong \overline{C}(\Gamma)$ in $\textit{hmf}_\omega.$
\end{lemma}

\begin{lemma} For any disjoint union of webs $\Gamma_1 \cup \Gamma_2$ there is a canonical isomorphism in $\textit{hmf}_\omega,$ namely
$\overline{C}(\Gamma_1 \cup  \Gamma_2) \cong \overline{C}( \Gamma_1) \otimes_{\mathbb{Q}[a,h]}\overline{C}( \Gamma_2).$ In particular we have
$$\overline{C}(\Gamma \cup \raisebox{-5pt}{\includegraphics[height=.2in]{orienloop.pdf}}) \cong \overline{C}(\Gamma)\brak{1} \otimes_{\mathbb{Q}[a,h]} \mathcal{A}.$$
\end{lemma}

Note that multiplication by $\partial_i \omega : = \partial \omega / \partial {x_i}$ endomorphism of $\overline{C}(\Gamma)$ is homotopic to zero. Moreover, multiplication by any polynomial in $(\textbf{a}, \textbf{b})$ induces a null-homotopic endomorphism of $\overline{C}(\Gamma)$ (see~\cite[Proposition 2]{KhR1}).   

\textit{Excluding a variable.} An important tool introduced in~\cite{KhR1} is the process of ``excluding a variable''. Supposed that $x_i$ is one of the generators of the polynomial ring$R$ and that $\omega = \sum a_i b_i \in R',$ where $R = R'[x_i].$ We say that $x_i$ is an \textit{internal} variable. Any $(R, \omega)$-factorization $(\textbf{a, b})$ restricts to an infinite rank factorization over $R'$, and we denote it by $(\textbf{a, b})'.$ Suppose furthermore that for some $j,$ $b_j= x_i -\alpha$ where $\alpha \in R'.$  Denote by $(\textbf{a'}, \textbf{b'})$ the factorization over $R'$ obtained from $(\textbf{a, b})$ by removing the $j$-th row and substituting $\alpha$ for $x_i$ everywhere in all other rows.
 \begin{lemma}
 Factorizations $(\textbf{a, b})'$ and $(\textbf{a'}, \textbf{b'})$ are isomorphic in the homotopy category of $(R', \omega)$-factorizations. \end{lemma}
 
 If $\Gamma$ is a closed web, the potential $\omega = 0$ and $\overline{C}(\Gamma)$ is a 2-periodic complex, thus we can take the cohomology of the corresponding complex. In particular, to the basic closed web with two vertices and with arcs labeled by $x_1$ and $x_2$ we assign the factorization over $\mathbb{Q}[a, h]$ with trivial potential, which is the quotient of $\overline{C}(\Gamma^1)$ by the relations $x_1 = x_4$ and  $x_2 = x_3.$ We obtain a complex with homology only in degree zero:
\[ H^0(\overline{C}(\Gamma^1)/_{x_1 = x_4,x_2 = x_3}) = \mathbb{Q}[a,h,x_1, x_2]/(\overline{u}^*_1, \overline{u}^*_2) \{-1\}, \, \text{where}\]
\begin{align*}
 \overline{u}^*_1 &= \overline{u}_1/ _{x_1=x_4, x_2 =x_3} = 3(x_1 + x_2)^2 -3x_1x_2 -3h(x_1 + x_2) -3a \\ \overline{u}^*_2 &=  \overline{u}_2/ _{x_1=x_4, x_2 =x_3} = -3(x_1 + x_2 -h).
\end{align*}
Therefore $H^0(\raisebox{-8pt}{\includegraphics[height=0.3in]{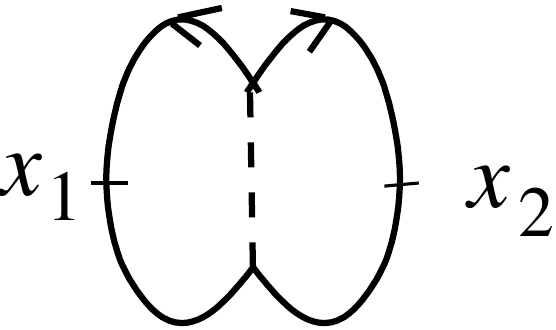}}) \cong \mathbb{Q}[a,h,x_1, x_2]/(x_1 + x_2 -h,\, x_1x_2 + a) \{-1 \},$ or equivalently
\[ H^0(\raisebox{-8pt}{\includegraphics[height=0.3in]{closed-web.pdf}}) \cong \mathbb{Q}[a,h,x_1]/(x_1^2 -h x_1- a) \{-1 \} \cong \mathcal{A} \quad \text{and} \quad H^1(\raisebox{-8pt}{\includegraphics[height=0.3in]{closed-web.pdf}}) = 0.\]

Similarly, the quotient of $\overline{C}(\Gamma^0)$ by the relations $x_1 = x_4$ and  $x_2 = x_3$ is a 2-complex with homology only in degree zero:
\[ H^0(\overline{C}(\Gamma^0)/_{x_1 = x_4,x_2 = x_3}) = \mathbb{Q}[a,h,x_1, x_2]/(\overline{\pi}^*_{41}, \overline{\pi}^*_{32}) \{-2\}, \, \text{where}\]
\begin{align*}
 \overline{\pi}^*_{41} &= \overline{\pi}_{41}/ _{x_1=x_4, x_2 =x_3} = 3(x_1^2 -h x_1 - a) \\ \overline{\pi}^*_{32} &=  \overline{u}_2/ _{x_1=x_4, x_2 =x_3} = -3(x_2^2 -h x_2 -a).
\end{align*}
Therefore $H^0(\raisebox{-8pt}{\includegraphics[height=0.3in]{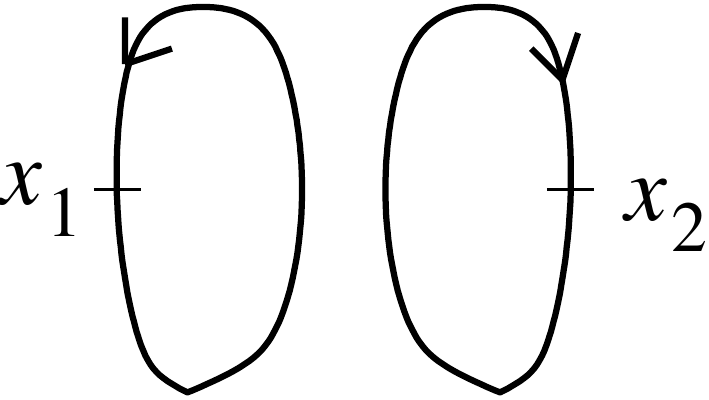}}) \cong \mathbb{Q}[a,h,x_1, x_2]/(x_1^2 +h x_1 -a,\, x_2^2 - hx_2 - a) \{-2 \}.$

Equivalently,  $H^0(\raisebox{-8pt}{\includegraphics[height=0.3in]{orcircles.pdf}}) \cong \mathcal{A} \otimes \mathcal{A} \quad \text{and} \quad H^1(\raisebox{-8pt}{\includegraphics[height=0.3in]{orcircles.pdf}}) = 0.$

%%%%%%%%%%%%%%%%%%%%%%%%%%%%%%%%%%%%%%%%%%%%%%%%%%%
%MAPS
%%%%%%%%%%%%%%%%%%%%%%%%%%%%%%%%%%%%%%%%%%%%%%%%%%%
 \subsection{Maps $\Lambda_0$ and $\Lambda_1$}\label{sec:maps}
The maps $\Lambda_0$ and $\Lambda_1$ between $\overline{C}(\Gamma^0)$ and $\overline{C}(\Gamma^1),$ where $\Gamma^0$ and $\Gamma^1$ are the web diagrams from Figure~\ref{maps}, are given by a pair of $2 \times 2$ matrices $(U_0, U_1)$ and $(V_0, V_1),$ respectively. 
\[U_0 = 
\left(\begin{array}{cc}x_4 - x_2 & 0 \\ x_1 - x_2 + x_3 + 2x_4 -\frac{3}{2} h & 1\end{array}\right), \quad U_1 = \left(
\begin{array}{cc}x_4 & -x_2 \\ -1 & 1\end{array}\right)
\]
 
\[V_0 = 
\left(\begin{array}{cc}1 & 0 \\ -x_1 + x_2 - x_3 - 2x_4 +\frac{3}{2}h & x_4-x_2\end{array}\right), \quad V_1 = \left(
\begin{array}{cc}1 & x_2 \\ 1 & x_4\end{array}\right)
\]
 
The compositions $\Lambda _0 \Lambda_1$ and$\Lambda_1 \Lambda_0$ are homotopic to the multiplication by $x_4 - x_2$ endomorphism of $\overline{C} (\Gamma^0)$ and $\overline{C}(\Gamma^1),$ respectively. This is easily seen from the following relations:
 \[U_0V_0 = U_1V_1 =(x_4 -x_2)\id , \quad  V_0U_0 = V_1U_1 = (x_4 -x_2) \id.\]
Thus $\Lambda_1 \circ \Lambda_0 = m(x_4-x_2)$ and $\Lambda_0 \circ \Lambda_1 = m(x_4-x_2)$. On the other hand, since the endomorphism of $\overline{C}(\Gamma^0)$---or $\overline{C}(\Gamma^1)$--- given by the multiplication by $x_1 + x_2 - x_3 -x_4$ is null-homotopic, we also have $\Lambda_1 \circ \Lambda_0 =  m(x_1 - x_3)$ and  $\Lambda_0 \circ \Lambda_1 = m(x_1 - x_3).$ Both maps $\Lambda_0$ and $\Lambda_1$ are maps of degree 1.

Considering the Koszul matrices for $\overline{C}(\Gamma^0)$ and $\overline{C}(\Gamma^1),$ we apply certain row transformation to each of them:

$ \overline{C}(\Gamma^0): \quad    \left(\begin{array}{cc}
\overline{\pi}_{41}  & x_1-x_4 \\ 
\overline{\pi}_{32}  & x_2-x_3\end{array} \right) \stackrel{[21]_{-1}}{\longrightarrow}  \left(\begin{array}{cc}
\overline{\pi}_{41}  & x_1 + x_2 -x_3 -x_4 \\ 
\overline{\pi}_{32} -\overline{\pi}_{41}  & x_2-x_3\end{array} \right) $ 

$ \overline{C}(\Gamma^1): \quad    \left(\begin{array}{cc}
\overline{u}_1  & x_1 + x_2 - x_3 -x_4 \\ 
\overline{u}_2  & x_1x_2-x_3x_4\end{array} \right) \stackrel{[12]_{x_2}}{\longrightarrow}  \left(\begin{array}{cc}
\overline{u}_1 + x_2 u_2 & x_1 + x_2 -x_3 -x_4 \\ 
\overline{u}_2 &(x_4 - x_2)(x_2-x_3)\end{array} \right).$  

We apply further a twist to obtain 
\[ \overline{C}(\Gamma^0) \stackrel {\cong}{\longrightarrow} \left(\begin{array}{cc}
\overline{\pi}_{41} -k (x_2-x_3)  & x_1 + x_2 -x_3 -x_4 \\ 
\overline{\pi}_{32} -\overline{\pi}_{41} + k (x_1 + x_2 -x_3 -x_4 ) & x_2-x_3\end{array} \right)\]

\[\overline{C}(\Gamma^1) \stackrel{\cong}{\longrightarrow} \left(\begin{array}{cc}
\overline{u}_1 + x_2 u_2 + 2 (x_4 - x_2)(x_2-x_3) & x_1 + x_2 -x_3 -x_4 \\ 
\overline{u}_2 -2 (x_1 + x_2 -x_3 -x_4) & (x_4 - x_2)(x_2-x_3)\end{array} \right) \]

where $k = x_1 + x_2 + x_3 -\frac{3}{2}h.$ An easy computation shows that the Koszul matrices above, hence $\overline{C}(\Gamma^0)$ and $\overline{C}(\Gamma^1)$, have the following equivalent forms 

$ \overline{C}(\Gamma^0) \cong \left(\begin{array}{cc}
a_1  & x_1 + x_2 -x_3 -x_4 \\ 
a_2 (x_4 -x_2) & x_2-x_3\end{array} \right) \, \text{and} \,\,  \overline{C}(\Gamma^1) \cong \left(\begin{array}{cc}
a_1  & x_1 + x_2 -x_3 -x_4 \\ 
a_2 & (x_4 -x_2)(x_2-x_3)\end{array} \right) $

where $a_1 = \overline{\pi}_{41} - \overline{\pi}_{21} + \overline{\pi}_{31}$ and $a_2 = -2x_1 -2x_2 -x_3 -x_4 +3h.$ The first rows are identical and the second rows are related. Consider the \textit{flip} homomorphisms $$\psi_{x_4-x_2}\co (a_2,  (x_4-x_2)(x_2-x_3)) \longrightarrow (a_2(x_4-x_2), x_2-x_3)$$
\begin{displaymath} 
\xymatrix @C=23mm@R=13mm{
R \ar [r]^{a_2} \ar [d]_{ 1} & R \ar [d]_{x_4-x_2} \ar[r]^{(x_4-x_2)(x_2-x_3)} & R \ar [d]_1 \\
R \ar [r]^{a_2(x_4-x_2)} & R \ar [r]^{x_2-x_3}   & R}
\end{displaymath}
and $ \psi'_{x_4-x_2} \co (a_2(x_4-x_2), x_2-x_3) \longrightarrow (a_2, (x_4-x_2)(x_2-x_3))$

\begin{displaymath} 
\xymatrix @C=23mm@R=13mm{
R \ar [r]^{a_2(x_4-x_2)} \ar [d]_{ x_4-x_2} & R \ar [d]_1 \ar[r]^{x_2-x_3} & R \ar [d]_{x_4-x_2} \\
R \ar [r]^{a_2} & R \ar [r]^{(x_4-x_2)(x_2-x_3)}   & R}
\end{displaymath}
Using the equivalent Koszul matrices for $\overline{C}(\Gamma^0), \overline{C}(\Gamma^1),$ the following lemma follows. 
\begin{lemma}
$\Lambda_0 = \id \otimes \,\psi'_{x_4-x_2}$ and $\Lambda_1 = \id \otimes \,\psi_{x_4-x_2}.$
\end{lemma}

%%%%%%%%%%%%%%%%%%%%%%%%%%%%%%%%%%%%%%%%%%%%%%%%%%%%%Complexes
\subsection{Complexes of factorizations}\label{sec:complexes}
We start with a generic diagram $D$ of a tangle $T$ with boundary points $B$, and put (at least) one mark on each segment bounded by two crossings. We let $m(D)$ be the set of all marks of $D,$ and consider the polynomial ring $R = \mathbb{Q}[a, h, x_i]$ for all $i \in m(D),$ and its subring $R' = \mathbb{Q}[a, h, x_i]$ for all $i \in B.$   

We associate to each crossing $p$ in $D$ a complex $\overline{C}_p$ of matrix factorizations, as explained in Figure~\ref{fig:crossings to chain complexes}, where the underlined objects are at the cohomological degree $0.$ 
  \begin{figure}[ht]
\[ \raisebox{-13pt}{\includegraphics[height=.45in]{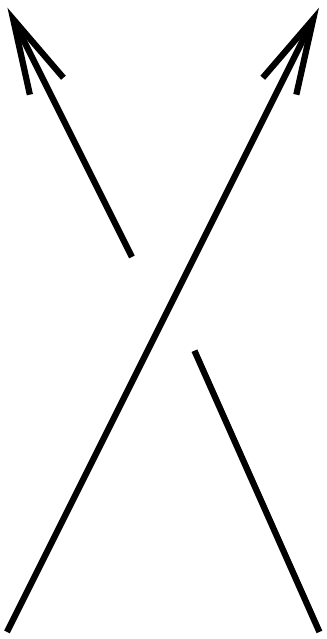}} \,\, = \,\, \left[\,\, 0 \longrightarrow \raisebox{-13pt}{\includegraphics[height=.45in]{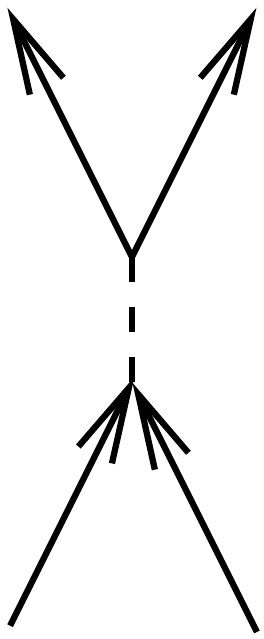}}\,\,\{2\} \longrightarrow \underline{\raisebox{-13pt}{\includegraphics[height=.45in]{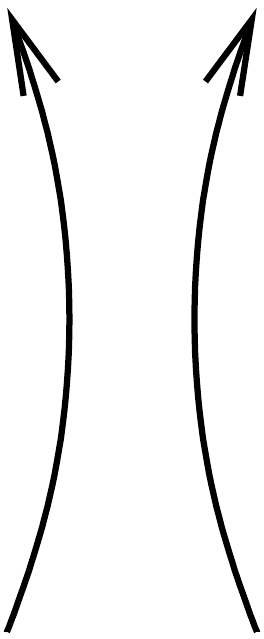}}\,\, \{1\}}\longrightarrow 0\,\, \right ]\]

\[ \raisebox{-13pt}{\includegraphics[height=.45in]{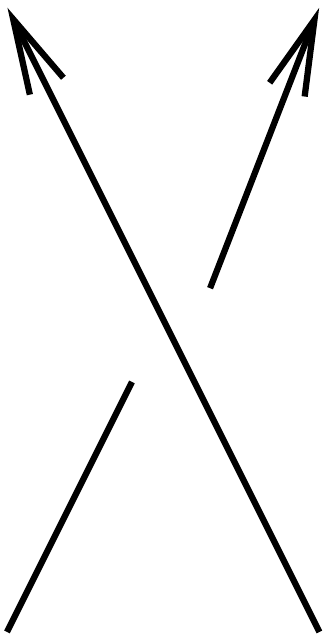}} \,\, = \,\, \left[\,\, 0 \longrightarrow \underline{\raisebox{-13pt}{\includegraphics[height=.45in]{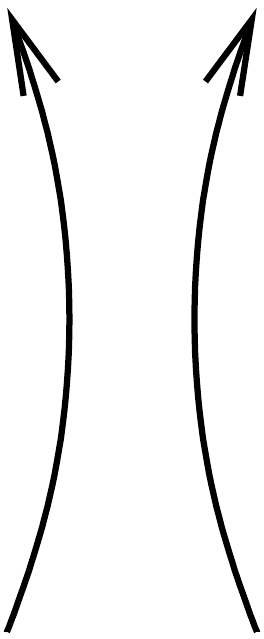}}\,\,\{-1\}} \longrightarrow\raisebox{-13pt}{\includegraphics[height=.45in]{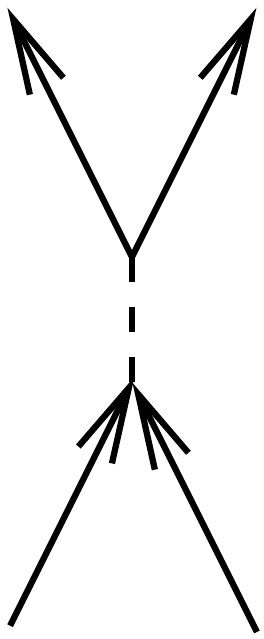}}\,\, \{-2\}\longrightarrow 0\,\, \right ]\]
\caption{Complex applied to a crossing}
\label{fig:crossings to chain complexes}
\end{figure}

Precisely, we have
\[ \raisebox{-8pt}{\includegraphics[height=.3in]{poscros.pdf}} = [\,0 \longrightarrow \overline{C}(\Gamma^1)\{2\} \stackrel{\Lambda_1}{\longrightarrow} \overline{C}(\Gamma^0)\{1\} \longrightarrow 0\, ]\]
\[ \raisebox{-8pt}{\includegraphics[height=.3in]{negcros.pdf}} = [\, 0 \longrightarrow \overline{C}(\Gamma^0)\{-1\} \stackrel{\Lambda_0}{\longrightarrow} \overline{C}(\Gamma^1)\{-2\} \longrightarrow 0\,] \]
where $\Gamma^0, \Gamma^1$ are the oriented and the singular resolutions from Figure~\ref{fig:resolutions}, and where the matrix factorization $\overline{C}(\Gamma^0)$ is at the cohomological degree $0.$

We associate to $D$ a complex of factorizations, which we denote it by $C(D),$  and which is the tensor product of $\overline{C}_p,$ over all crossings $p$ in $D,$ of $\overline{L}_i^j,$ over all arcs $i \to j,$ and of $\mathcal{A} \brak{1},$ over all oriented loops in $D$ with no crossings and no marks. The tensoring is done over appropriate polynomial rings so that $C(D)$ is a free $R$-module of finite rank.

$C(D)$ is a complex of graded ($R, \omega$)-factorizations, where $\omega = \sum_{i \in B} o(i)p(a, h, x_i).$ We regard $C(D)$ as an object in $K_\omega : = Kom(\textit{hmf}_\omega),$ the homotopy category of complexes over $\textit{hmf}_\omega.$ Note that the differentials of $C(D)$ are grading-preserving. In Section \ref{sec:invariance} we show that  the isomorphism class of $C(D)$ is a tangle invariant.

If $T$ is a link, the set $B$ of boundary points is empty, $R' = \mathbb{Q}[a, h],$ and $C(D)$ is a complex of graded $\mathbb{Q}[a, h]$-modules. 

%%%%%%%%%%%%%%%%%%%%%%%%%%%%%%%%%%%%%%%%%%%%%%%%%%%%%ISOMORPHISMS
%%%%%%%%%%%%%%%%%%%%%%%%%%%%%%%%%%%%%%%%%%%%%%%%%%
\subsection{Isomorphisms}\label{sec:isomorphisms}
We show that $\overline{C}(\Gamma)$ mimics the skein relations of Figure~\ref{fig:skein relations}. 
\begin{proposition}\label{isom 1}(\textit{First Isomorphism}) There is an isomorphism in $hmf_\omega$:
\[\overline{C}(\raisebox{-8pt}{\includegraphics[height=.3in]{singresclosed.pdf}}\,) \cong \overline{C}(\raisebox{-8pt}{\includegraphics[height=.3in]{arc.pdf}}\,)\brak{1}.\]
\end{proposition}

\begin{proof} Consider the webs $\Gamma$ and $\Gamma_1$ given in Figure~\ref{fig:direct sum decomposition 1-webs}. Factorizations $\overline{C}(\Gamma),\overline{C}(\Gamma_1)$ are $(R', \omega)$-factorizations (the later has infinite rank over $R'$), where $R' = \mathbb{Q}[a, h, x_2,x_3]$ and $\omega = p(a, h, x_2) - p(a, h, x_3).$
\begin{figure}[h]
\[ \Gamma = \raisebox{-27pt}{\includegraphics[height=.8in]{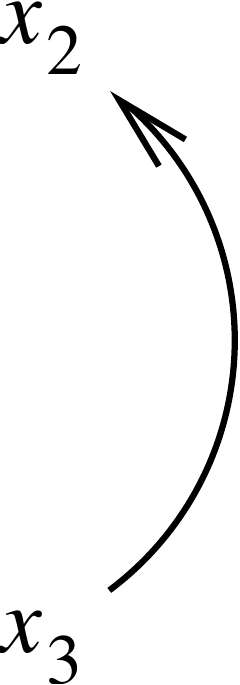}} \hspace{3cm} \Gamma_1 = \raisebox{-27pt}{\includegraphics[height=.8in]{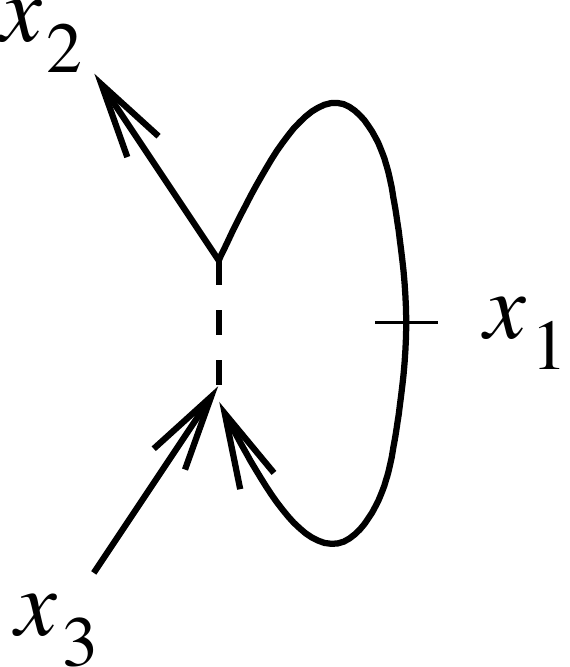}}\]
\caption{Webs $\Gamma$ and  $\Gamma_1$}
\label{fig:direct sum decomposition 1-webs}
\end{figure}

In Koszul form, $\overline{C}(\Gamma_1) = \left (\begin{array}{cc}\overline{u}'_1 & x_2-x_3 \\
\overline{u}'_2 & x_1(x_2 - x_3)  \end{array} \right)$ where $\overline{u}'_1 = \overline{u}_1(x_2, x_1,x_1, x_3)$ and $\overline{u}'_2 = \overline{u}_1(x_2, x_1,x_1, x_3).$ We apply the elementary row operation $[12]_{x_1}$ and obtain:
$$\left (\begin{array}{cc}\overline{u}'_1 & x_2-x_3 \\
\overline{u}'_2 & x_1(x_2 - x_3)  \end{array} \right) \stackrel{\cong}{\longrightarrow} \left (\begin{array}{cc}\overline{u}'_1+ x_1 \overline{u}'_2 & x_2-x_3 \\
\overline{u}'_2 & 0  \end{array} \right)$$
An easy computation shows that $\overline{u}'_1+ x_1 \overline{u}'_2= \overline{\pi}_{23}$ and $\overline{u}'_2 = -3(x_1+x_3) + 3h,$ thus we have $\overline{C}(\Gamma_1) \brak{1}\cong \left (\begin{array}{cc}\overline{\pi}_{23} & x_2-x_3 \\
0 & -\overline{u}'_2 \end{array} \right).$ Since $x_1$ is an internal variable, we can eliminate it by removing the row $(0 \quad -\overline{u}'_2) .$ Therefore, $\overline{C}(\Gamma_1) \brak {1} \cong (\overline{\pi}_{23}, \, x_2-x_3 ) = \overline{C}(\Gamma).$ \end{proof}

\begin{proposition}\label{isom 2}(\textit{Second Isomorphism})
There is an isomorphism in the category $hmf_{\omega}$:
\[\overline{C}(\raisebox{-13pt}{\includegraphics[height=.45in]{removedweb1.pdf}}) \cong \overline{C}(\raisebox{-13pt}{\includegraphics[height=.45in]{removedweb2.pdf}})\{-1\} \oplus \overline{C}(\raisebox{-13pt}{\includegraphics[height=.45in]{removedweb2.pdf}})\{1\}.\]
\end{proposition}

\begin{proof} The proof is the same as that of the ``Direct sum decomposition II'' in~\cite{KhR1}.
\end{proof}

\begin{proposition}\label{isom 3}(\textit{Third Isomorphism})
There is an isomorphism in the category $hmf_{\omega}$:
\[\overline{C}(\, \raisebox{-8pt}{\includegraphics[height=.3in]{2singresol.pdf}}\,) \cong \overline{C}(\,\raisebox{-5pt}{\includegraphics[height=.25in]{2arcs.pdf}}\,).\]
\end{proposition}
\begin{proof} Consider the webs $\Gamma$ and $\Gamma'$ depicted in Figure~\ref{fig:direct sum decomposition 3-webs}. Factorizations $\overline{C}(\Gamma),\overline{C}(\Gamma')$ are $(R', \omega)$-factorizations, where $R' = \mathbb{Q}[a, h, x_1, x_2, x_3, x_4]$ and $\omega = p(a, h, x_1) - p(a, h, x_2) + p(a, h, x_3) - p(a, h, x_4).$
\begin{figure}[h]
\[ \Gamma = \raisebox{-20pt}{\includegraphics[height=.6in]{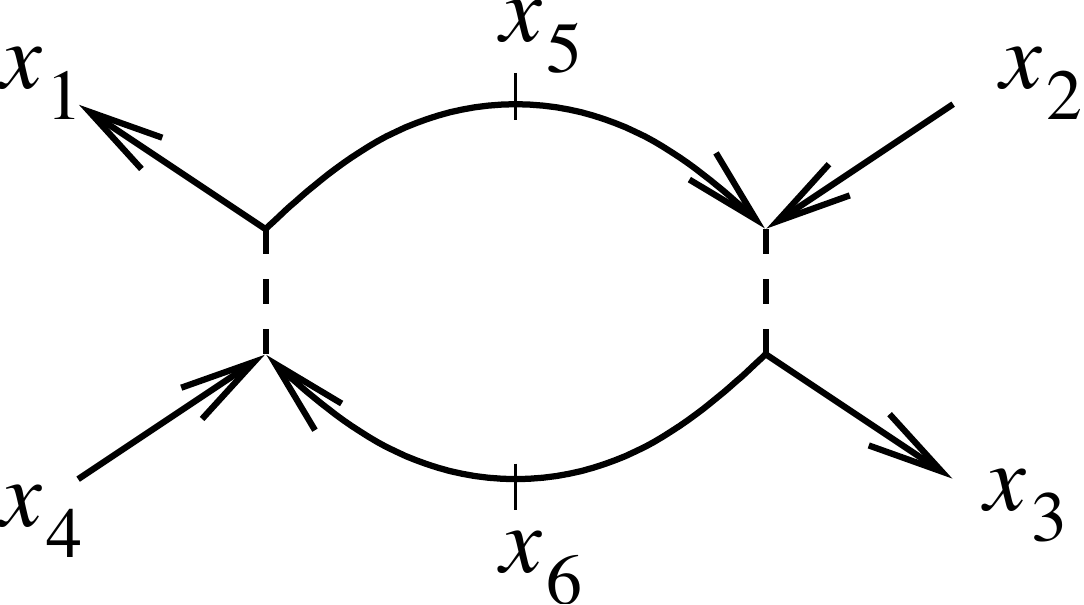}} \hspace{3cm} \Gamma' = \raisebox{-20pt}{\includegraphics[height=.6in]{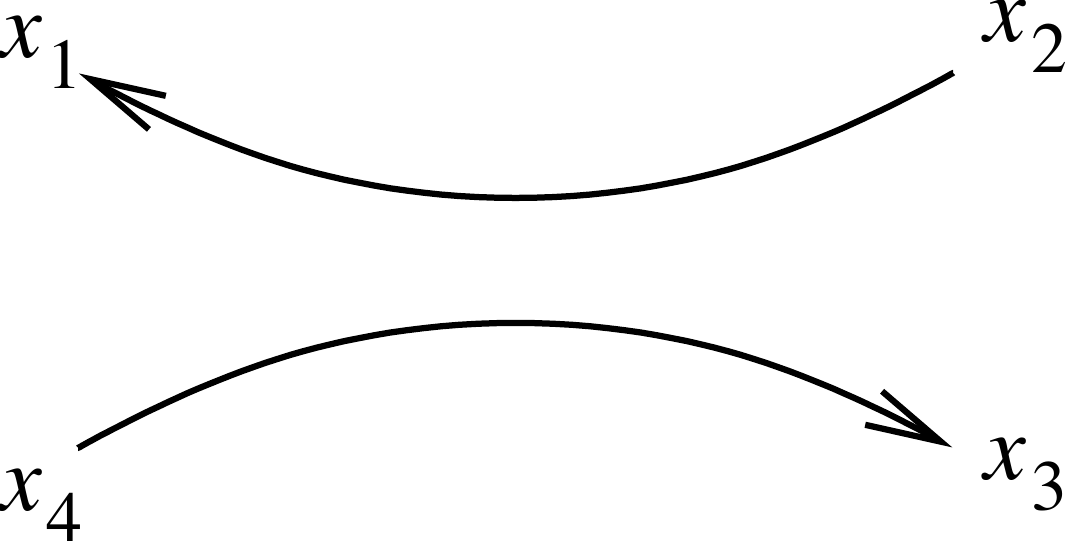}}\]
\caption{Webs $\Gamma$ and  $\Gamma'$}
\label{fig:direct sum decomposition 3-webs}
\end{figure}

In Koszul form, $\overline{C}(\Gamma) = (\textbf{a, b})\{-2\} = \left(\begin{array}{cc}\overline{u}_1'  & x_1 + x_5 -x_4 -x_6 \\ \overline {u}_2' & x_1x_5 -x_4x_6\\ \overline{u}_1'' & x_3 + x_6 -x_2 -x_5 \\  \overline{u}_2'' & x_3x_6 -x_2x_5  \end{array} \right )\{-2\},$ 

where $\overline{u}_i' = \overline{u}_i(x_1, x_5, x_6, x_4)$ and $\overline{u}_i''= \overline{u}_i(x_3, x_6, x_5, x_2), i = 1, 2.$ A shift of $\{-1\}$ corresponds to the second and fourth row. The potential $\omega$ lives in $R',$ thus $x_5$ and $x_6$ are internal variables. Knowing that $\overline{u}_2' = -3(x_4 + x_6) + 3h$ and $\overline{u}_2''= -3(x_2 + x_5) + 3h,$ we can exclude $x_5, x_6$ by crossing out the second and fourth row and replacing $x_5 = h-x_2$ and $x_6 = h-x_4$ in the first and third row of $(\textbf{a, b}).$ In particular, $\overline{C}(\Gamma)$ is isomorphic to the matrix factorization with Koszul form
$$(\textbf{c, d}) = \left(\begin{array}{cc}\overline{u}_1(x_1, h -x_2, h-x_4, x_4) & x_1 -x_2 \\ \overline{u}_1(x_3, h-x_4,h- x_2, x_2) & x_3 -x_4\end{array} \right). $$

An easy computations shows that $\overline{u}_1(x_1, h -x_2,h- x_4, x_4) = \overline{\pi}_{21}$ and $\overline{u}_1(x_3, h -x_3, h-x_2, x_2) = \overline{\pi}_{43}$, thus $\{\textbf{c, d}\} = \overline{C}(\Gamma'),$ which implies that $\overline{C}(\Gamma) \cong \overline{C}(\Gamma').$
\end{proof}

\begin{proposition}\label{isom 4}(\textit{Fourth Isomorphism})
There are isomorphisms in the category $hmf_{\omega}$:
\[\overline{C}(\, \raisebox{-15pt}{\includegraphics[height=.5in]{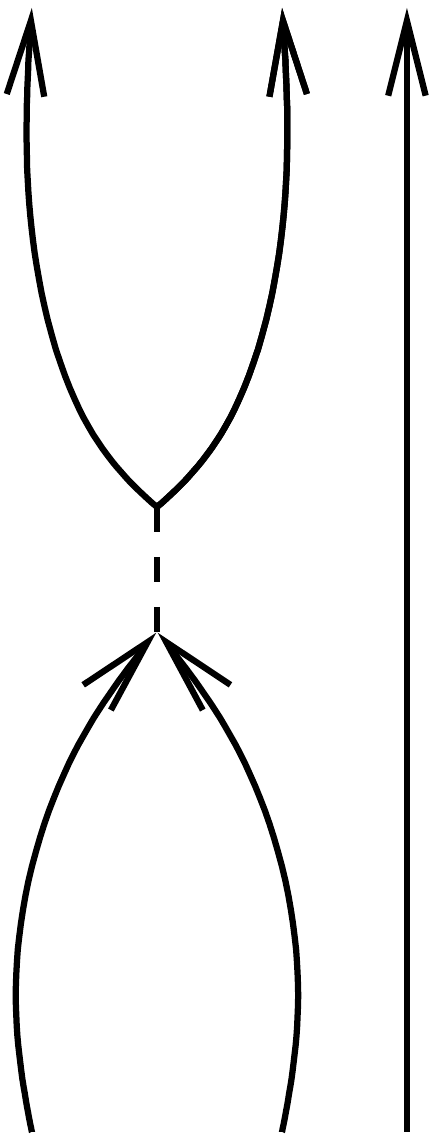}}\,) \cong \overline{C}(\,\raisebox{-15pt}{\includegraphics[height=.5in]{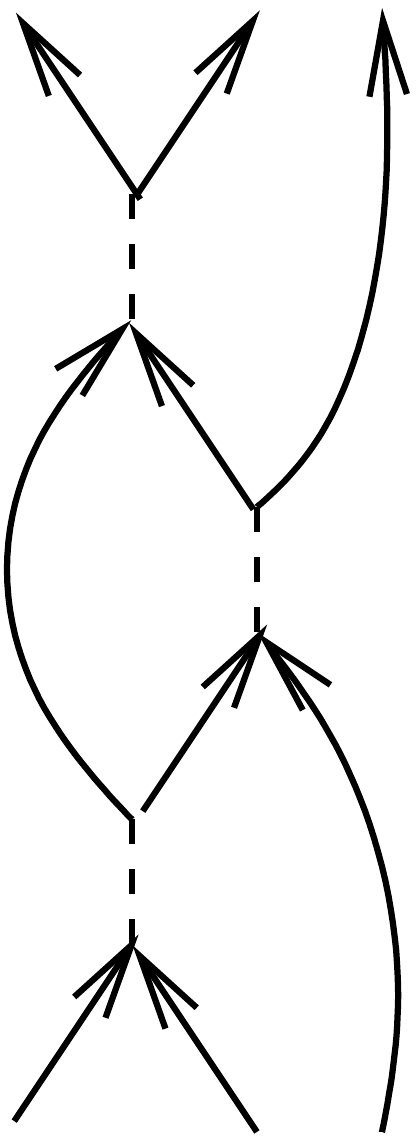}}\,) \quad \text{and} \quad \overline{C}(\, \raisebox{-15pt}{\includegraphics[height=.5in]{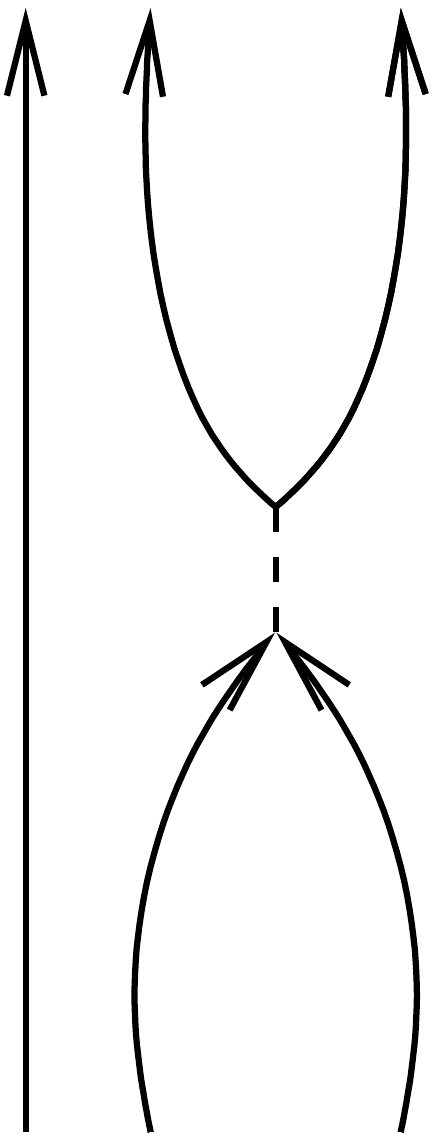}}\,) \cong \overline{C}(\,\raisebox{-15pt}{\includegraphics[height=.5in]{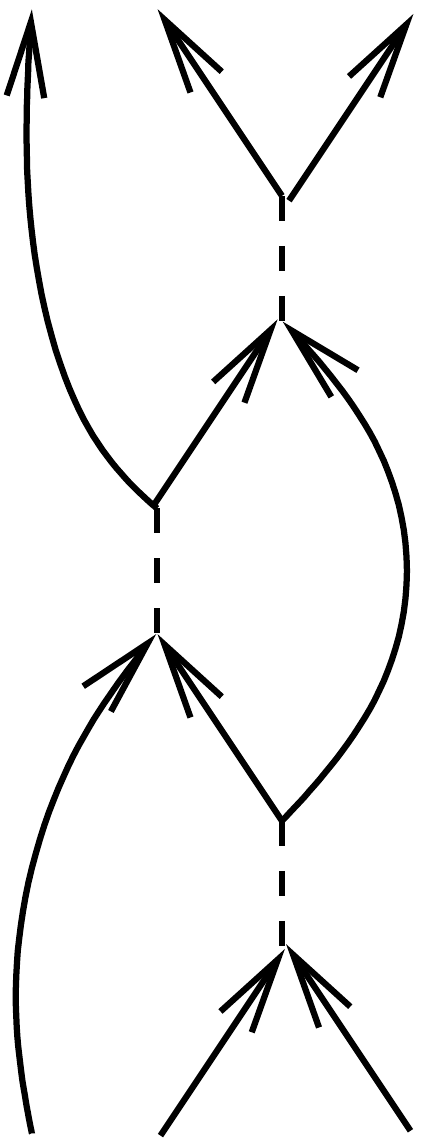}}\,).\]
\end{proposition}

\begin{proof}
Consider the webs $\Gamma_1$ and $\Gamma_2$ depicted in Figure~\ref{fig:direct sum decomposition 4-webs}. This time, the potential has the form $\omega = \sum_{i \in \{1,2,3\}}p(a, h, x_i) - \sum_{j \in \{4,5,6\}}p(a, h, x_j) \in R',$ where  $R' = \mathbb{Q}[a, h, x_1, x_2, x_3, x_4, x_5, x_6].$
\begin{figure}[h]
\[ \Gamma_1 = \raisebox{-42pt}{\includegraphics[height=1.2in]{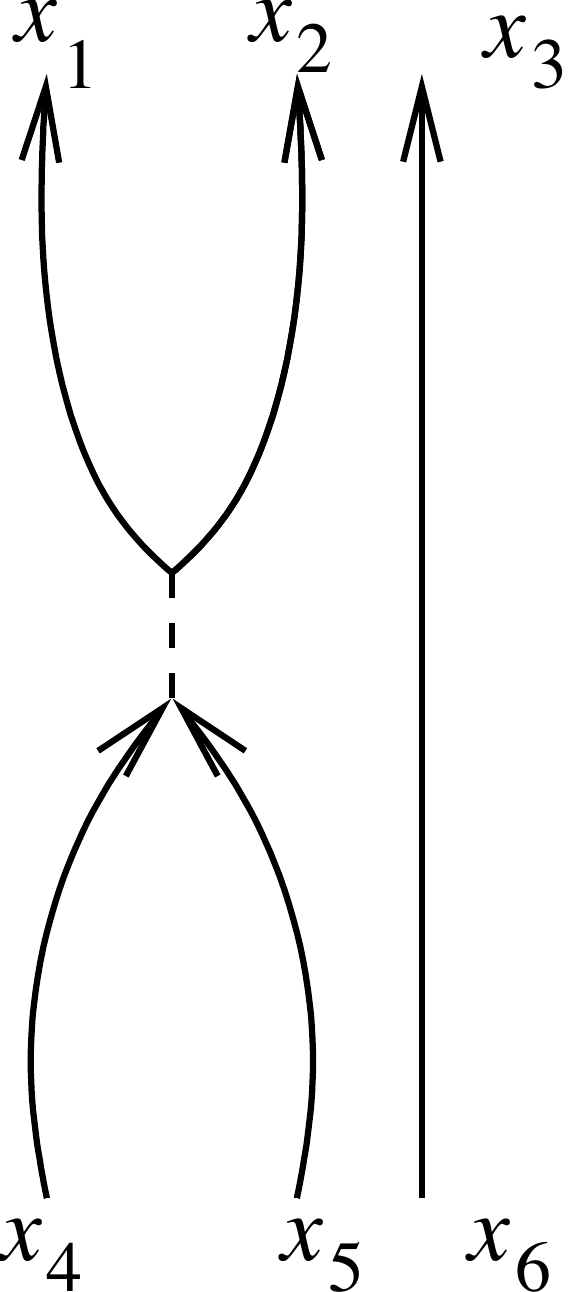}} \hspace{3cm} \Gamma_2 = \raisebox{-42pt}{\includegraphics[height=1.2in]{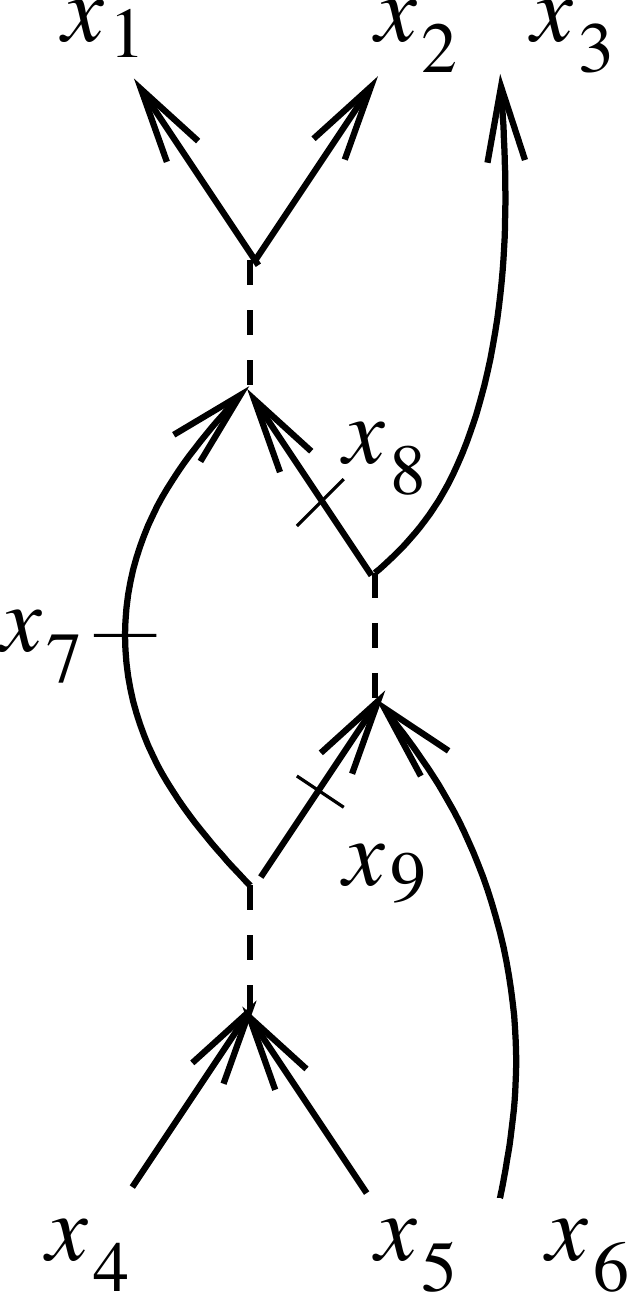}}\]
\caption{Webs $\Gamma_1$ and  $\Gamma_2$}
\label{fig:direct sum decomposition 4-webs}
\end{figure}

In Koszul form, $\overline{C}(\Gamma_2) = \left(\begin{array}{cc}\overline{u}_1'  & x_1 + x_2 -x_7 -x_8 \\ \overline {u}_2' & x_1x_2 -x_7x_8\\ \overline{u}_1'' & x_3 + x_8 -x_6 -x_9 \\  \overline{u}_2'' & x_3x_8 -x_6x_9 \\  \overline{u}_1''' & x_7 + x_9 -x_4 -x_5 \\  \overline{u}_2''' & x_7x_9 -x_4x_5  \end{array} \right )\{-3\},$ 

where $\overline{u}_i' = \overline{u}_i(x_1, x_2, x_8, x_7), \overline{u}_i''= \overline{u}_i(x_8, x_3, x_6, x_9)$ and $\overline{u}_i'''= \overline{u}_i(x_7, x_9, x_5, x_4)$ for $i = 1, 2.$ A shift by $\{-1\}$ was applied to the rows of $\overline{u}_2', \overline{u}_2''$ and $\overline{u}_2'''.$  Variables $x_7, x_8$ and $x_9$ are internal variables and we can use $\overline{u}_2' = -3(x_7 + x_8) + 3h$ and $\overline{u}_2''= -3(x_6 + x_9) + 3h$ to exclude $x_8$ and $x_9,$ by crossing out the second and fourth row and substituting in every other row $x_8 = h-x_7$ and $x_9 = h-x_6.$ To exclude the internal variable $x_7,$ we use the right-hand entry of the third row, which now has the form $-x_7 + x_3.$ After these operations, we obtain
$$\overline{C}(\Gamma_2) \cong \left(\begin{aligned}a_1 &= \overline{u}_1(x_1, x_2 ,h-x_3, x_3) \\  a_2 &= \overline{u}_1(x_3, h-x_6, x_5, x_4) \\ a_3 &= \overline{u}_2(x_3, h-x_6, x_5, x_4) \end{aligned}   \quad \begin{aligned} b_1 &= x_1 + x_2 -h \\ b_2 &= x_3 +h - x_6 -x_4 -x_5\\ b_3 &= x_3(h-x_6) - x_4 x_5 \end{aligned} \right )\{-1\}.$$
 We perform the row operation $[21]_{-1}$ and arrive at 
\[ \overline{C}(\Gamma_2)\{1\} \cong \left(\begin{array}{cc}a_1  & b_1 + b_2 \\ a_2 -a_1 & b_2 \\ a_3 & b_3 \end{array} \right).\] 
Now we apply a twist to the left-hand entries of the first two rows above, for $k= \frac{h}{2} -x_1 -x_2 -x_4 -x_5 + x_3 -x_6,$ to obtain
 \[ \left(\begin{array}{cc}a_1+ kb_2  & b_1 + b_2 \\ a_2 -a_1 -k(b_1 + b_2) & b_2\\ a_3 & b_3 \end{array} \right). \] 
 Performing the row operation $[23]_{x_3 -x_6}$ we get
 \[\overline{C}(\Gamma_2)\{1\} \cong \left(\begin{array}{cc}a_1+ kb_2  & b_1 + b_2 \\ a_2 -a_1 -k(b_1 + b_2)+(x_3-x_6)a_3 & b_2\\ a_3 & b_3 - (x_3 -x_6)b_2 \end{array} \right ) =  \left(\begin{array}{cc}a'_1& b'_1 \\ a'_2& b'_2\\ a'_3 & b'_3 \end{array} \right).\]
 The later Koszul factorization $(\textbf{a}', \textbf{b}')$ is isomorphic in $hmf_{\omega}$ to $\overline{C}(\Gamma_2)\{1\},$ and we apply to it a twist with $k' =-\frac{1}{3}$
  \[\overline{C}(\Gamma_2)\{1\} \cong \left(\begin{array}{cc}a'_1& b'_1 \\ a'_2& b'_2 - \frac{1}{3} a'_3\\ a'_3 & b'_3 + \frac{1}{3} a'_2 \end{array} \right),\]
followed by the row operation $[23]_{-(x_3 - x_6 - x_4 - x_5)}$ 
 \[\overline{C}(\Gamma_2)\{1\} \cong \left(\begin{array}{cc}a'_1& b'_1 \\ a'_2 -(x_3 -x_6 -x_4- x_5)a'_3& b'_2 - \frac{1}{3} a'_3\\ a'_3 & b'_3 + \frac{1}{3} a'_2 + (x_3 -x_6 -x_4- x_5)(b'_2 -\frac{1}{3}a'_3) \end{array}\right).\]
 
 Let's denote the previous Koszul matrix by $(\textbf{a}'', \textbf{b}'').$ Replacing the entries of $\textbf{b}''$ we have
\[\overline{C}(\Gamma_2)\{1\} \cong \left(\begin{array}{cc}a''_1& x_1+ x_2 +x_3 -x_4-x_5-x_6 \\ a''_2 & x_3-x_6\\ a''_3 & x_1x_2-x_4x_5 \end{array} \right).\]
 Finally, we apply the row operation $[21]_1$ followed by a twist with $k'' = -\frac{3}{2}h+x_1 + x_2 + 2x_4 + 2x_5 - x_3 + x_6,$ and we get
 \begin{align*}\overline{C}(\Gamma_2)\{1\} &\cong \left(\begin{array}{cc}a''_1& x_1+ x_2 -x_4-x_5 \\ a''_2 + a''_1 & x_3-x_6\\ a''_3 & x_1x_2-x_4x_5 \end{array} \right)\\
  &\cong  \left(\begin{array}{cc}a''_1 +k''(x_3-x_6) & x_1+ x_2 -x_4-x_5 \\ a''_2 + a''_1 - k''(x_1 + x_2 -x_4 -x_5) & x_3-x_6\\ a''_3 & x_1x_2-x_4x_5\end{array}\right).\end{align*}
 Computing the entries in the left column above, one founds
 \begin{align*}\overline{C}(\Gamma_2)\{1\} &\cong \left(\begin{array}{cc}\overline{u}_1(x_1, x_2, x_4, x_5) & x_1+ x_2 -x_4-x_5 \\ \overline{\pi}_{63} & x_3-x_6 \\ \overline{u}_2(x_1 ,x_2, x_4, x_5) & x_1x_2-x_4x_5 \end{array} \right)\\
 &\cong \left(\begin{array}{cc} \overline{\pi}_{63} & x_3-x_6 \\ \overline{u}_1(x_1, x_2, x_4, x_5) & x_1+ x_2 -x_4-x_5 \\ \overline{u}_2(x_1 ,x_2, x_4, x_5) & x_1x_2-x_4x_5 \end{array} \right) \cong \overline{C}(\Gamma_1)\{1\}.\end{align*}
The later matrix is the Koszul form of the factorization $\overline{C}(\Gamma_1)\{1\},$ thus $\overline{C}(\Gamma_2) \cong \overline{C}(\Gamma_1).$ 
 \begin{figure}[h]
\[ \Gamma_3 = \raisebox{-42pt}{\includegraphics[height=1.2in]{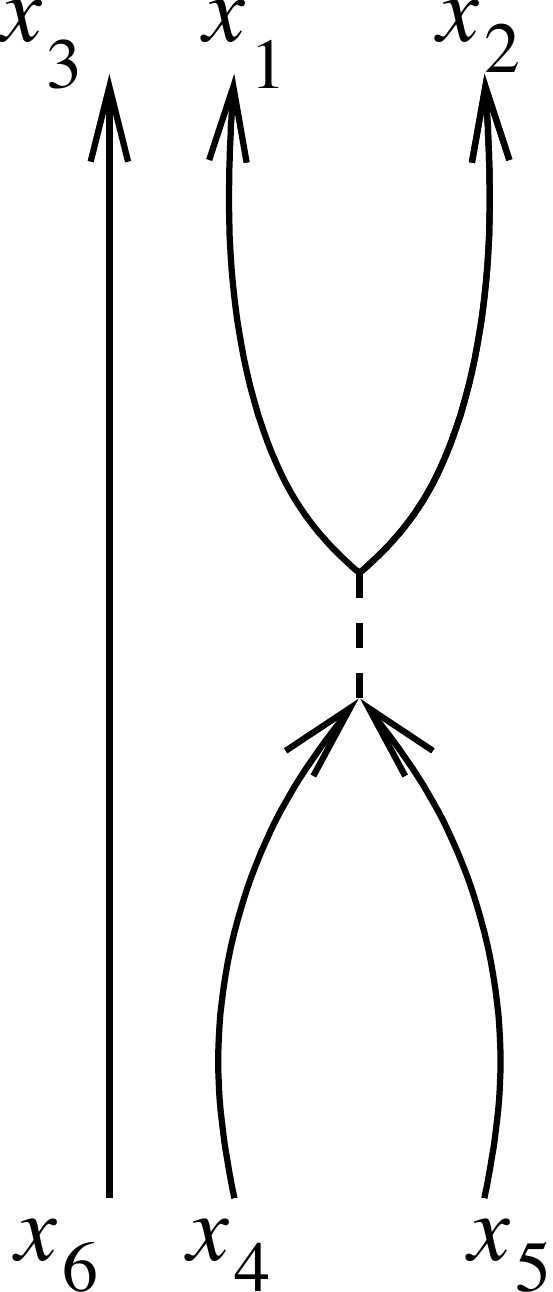}} \hspace{3cm} \Gamma_4 = \raisebox{-42pt}{\includegraphics[height=1.2in]{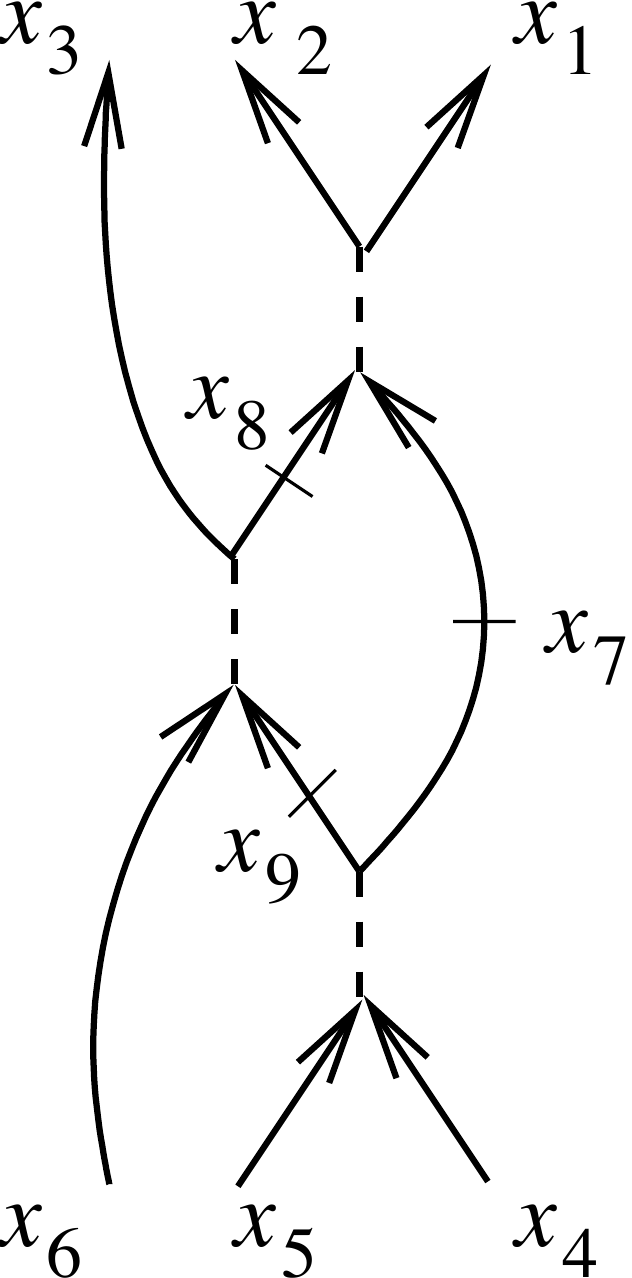}}\]
\caption{Webs $\Gamma_3$ and  $\Gamma_4$}
\label{fig:direct sum decomposition 4bis-webs}
\end{figure}

With the labeling of the diagrams $\Gamma_3, \Gamma_4$ given in Figure~\ref{fig:direct sum decomposition 4bis-webs}, the previous proof also implies that $\overline{C}(\Gamma_4) \cong \overline{C}(\Gamma_3)$ in $hmf_\omega.$
\end{proof}

\begin{definition}
Let $\Gamma$ be a closed web and $p(\Gamma)$ be the mod 2 number of circles in the modification of $\Gamma$ obtained by replacing all singular resolutions with the oriented resolution. Factorization $\overline{C}(\Gamma)$ is a 2-complex and has cohomology only in degree $p(\Gamma).$ Define the cohomology groups of $\overline{C}(\Gamma)$ as
\[ \overline{H}(\Gamma) : = H^{p(\Gamma)}(\overline{C}(\Gamma)). \]
\end{definition}

$\overline{H}(\Gamma)$ is a $\mathbb{Z}$-graded module over $\mathbb{Q}[a, h].$ The isomorphisms obtained in this section together with the fact that the skein relations in Figure~\ref{fig:skein relations} determine the evaluation of $\brak{\Gamma},$ for any web $\Gamma,$ imply the following result.

\begin{proposition}
For any closed web $\Gamma,$ the graded dimension of $\overline{H}(\Gamma)$ is $\brak {\Gamma},$ namely
\[ \sum _{j \in \mathbb{Z}} q^j \rk_{\mathbb{Q}[a,h]} \overline{H}^j(\Gamma) = \brak{\Gamma}.\]
\end{proposition}

\begin{remark}
Note that any resolution $\Gamma$ of a link diagram consists of a disjoint union of closed webs $\cup_k \Gamma_k,$  and its ``homology'' satisfies $\overline{H}(\Gamma) \cong \mathcal{A}^k.$ In Section~\ref{sec:TQFT} we show that $\overline{H}$ can be regarded as a $(1+1)$--dimensional TQFT functor.
\end{remark}

%%%%%%%%%%%%%%%%%%%%%%%%%%%%%%%%%%%%%%%%%%%%%%%%%%%%%INVARIANCE
%%%%%%%%%%%%%%%%%%%%%%%%%%%%%%%%%%%%%%%%%%%%%%%%%%%

\section{Invariance under Reidemeister moves}\label{sec:invariance}

\begin{theorem}\label{invariance}
If $D$ and $D'$ are two diagrams representing the same tangle $T,$ then complexes $C(D)$ and $C(D')$ are isomorphic in $K_{\omega}.$
\end{theorem}
\begin{proof}
\textbf{Reidemeister I.}
Consider diagrams $D$ and $\Gamma$ that differ only in a circular region as in the figure below:
$$D=\raisebox{-11pt}{\includegraphics[height=0.45in]{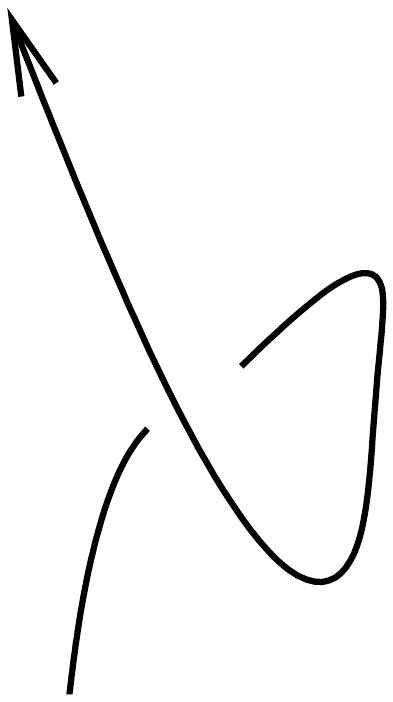}}\qquad
\Gamma=\raisebox{-13pt}{\includegraphics[height=0.45in]{arc.pdf}}$$
The complex $C(D)$ has the form
$$0\longrightarrow \underline{\overline{C}(\Gamma_2) \{-1\}} \stackrel{\Lambda_0}{\longrightarrow} \overline{C}(\Gamma_1) \{-2\} \longrightarrow 0$$ 
where
$$\Gamma_2 = \raisebox{-20pt}{\includegraphics[height=.6in]{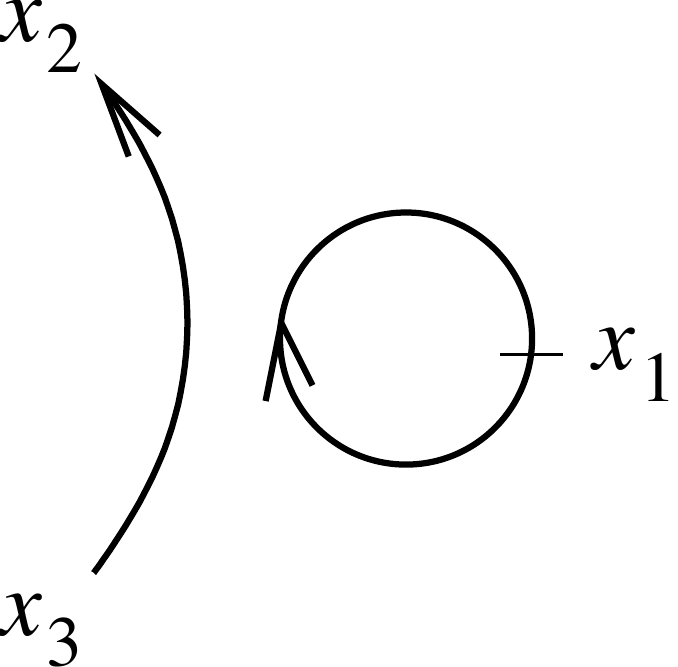}}\quad \mbox{and} \quad\raisebox{-22pt}{\includegraphics[height=0.65in]{lsingresclosed.pdf}} = \Gamma_1$$
Let $Y_1 \subset \overline{C}(\Gamma_2)\{-1\}$ be the inclusion of $\overline{C}(\Gamma)\brak{1}\{-2\}$ onto the first summand of $\overline{C}(\Gamma_2)\{-1\},$ and $Y_2 = f(\overline {C}(\Gamma)\brak{1}\{-2\}) \subset \overline{C}(\Gamma_2)\{-1\},$ where $f \co \overline {C}(\Gamma)\brak{1}\{-2\} \to \overline{C}(\Gamma_2)\{-1\}, f = \id \otimes[ m(x_1) \iota]$. Here we used that $\overline{C}(\Gamma_2) \cong \overline{C}(\Gamma)\brak{1} \otimes_{\mathbb{Q}[a,h]} \mathcal{A}.$

Thus $\overline{C}(\Gamma_2)\{-1\} \cong Y_1 \oplus Y_2$ in $hmf_\omega.$ On the other hand, from the First Isomorphism we know that $\overline{C}(\Gamma)\{-2\} \cong \overline{C}(\Gamma) \brak{1}\{-2\},$ therefore
 the complex $\overline{C}(D)$ is isomorphic in $K_{\omega}$ to the direct sum
\begin{align*} 
0 \longrightarrow &\underline{Y_2} \longrightarrow 0 \\ 
0 \longrightarrow &\underline{Y_1} \stackrel{\cong}{\longrightarrow} \overline{C}(\Gamma)\brak{1}\{-2\}. \end{align*} 
Since $Y_2 \cong \overline{C}(\Gamma)$ and the second summand is contractible, it implies that $C(D) \cong (\Gamma),$ in the category $K_{\omega}.$

A similar approach is used to prove the invariance under Reidemeister I involving a positive kink.

%%%%%%%%%%%%%%%%%%%%%%%%%%%%%%%%%%%%%%%%%%%%%%%%%%%
\textbf{Reidemeister IIa.} Consider diagrams $D_1$ and $\Gamma$ that differ only in a circular region as in the figure below:
$$D_1=\raisebox{-11pt}{\includegraphics[height=0.45in]{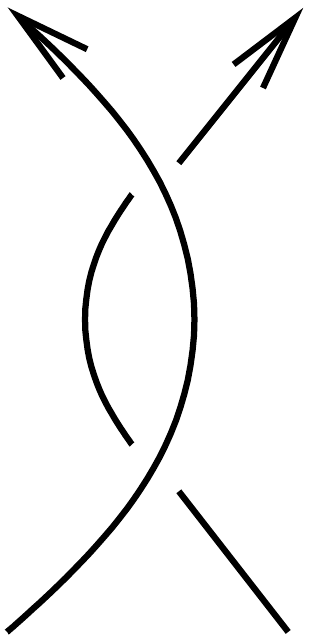}}\qquad
\Gamma=\raisebox{-13pt}{\includegraphics[height=0.45in]{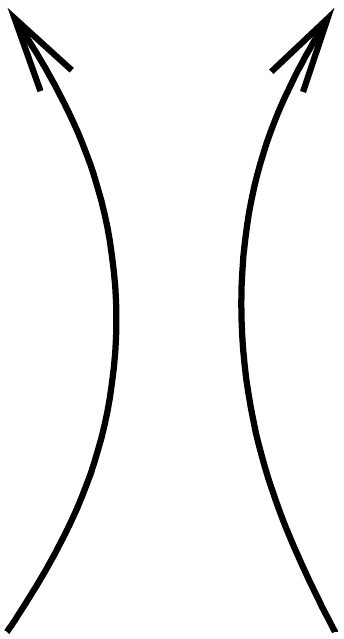}}$$

The complex $C(D_1)$ has the form
$$ 0\longrightarrow \overline{C}(\Gamma_{00})\{1\} \stackrel{(f_1, f_3)^t}{\longrightarrow} \underline{\overline{C}(\Gamma_{01}) \oplus \overline{C}(\Gamma_{10})} \stackrel{(f_2, - f_4)}{\longrightarrow} \overline{C}(\Gamma_{11})\{-1\} \longrightarrow 0,$$
whose objects are the matrix factorizations corresponding to the four resolutions of $D$ given in Figure~\ref{fig:reidIIA}, with potential $\omega = p(a, h, x_1) + p(a, h, x_2) - p(a, h, x_3)  - p(a, h, x_4).$

\begin{figure}[ht]
\raisebox{-11pt}{\includegraphics[height=2.2in]{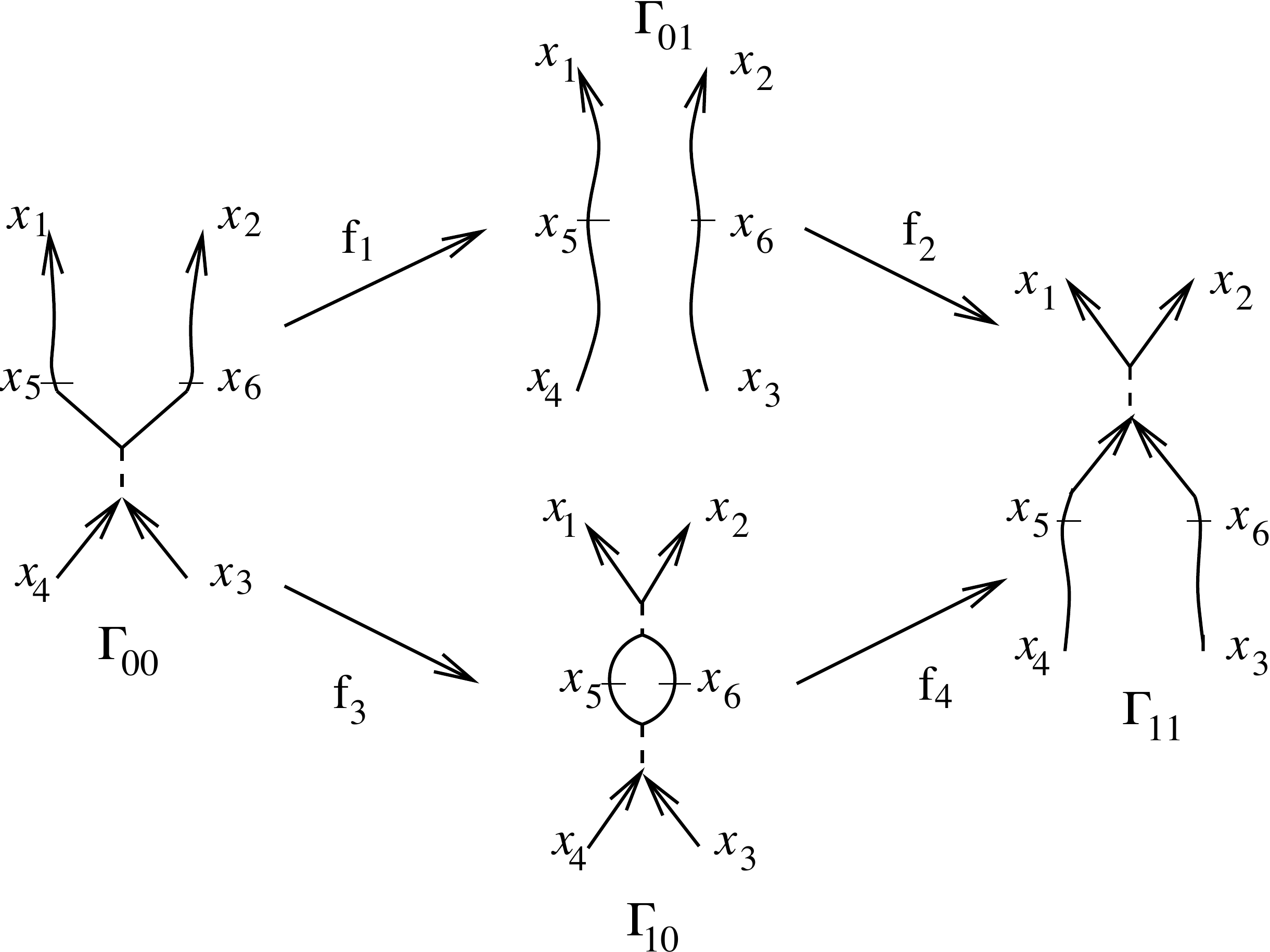}}
\caption{}\label{fig:reidIIA}
\end{figure}
Using the Second Isomorphism and that the marking doesn't matter, we have  
$$\overline{C}(\Gamma_{01}) \cong \overline{C}(\Gamma^0)$$ $$ \overline{C}(\Gamma_{00})\cong \overline{C}(\Gamma^1) \cong \overline{C}(\Gamma_{11})$$ $$ \overline{C}(\Gamma_{10})\cong \overline{C}(\Gamma^1)\{1\} \oplus \overline{C}(\Gamma^1)\{-1\},$$

where $\Gamma^0$ and $\Gamma^1$ are the diagrams from figure~\ref{maps}. Therefore the complex $C(D_1)$ is isomorphic (in $K_{\omega}$) to the complex
$$ 0 \longrightarrow \overline{C}(\Gamma^1)\{1\} \stackrel{(f_1, f_{03}, f_{13})^t}{\longrightarrow} \underline{\overline{C}(\Gamma^0) \oplus \overline{C}(\Gamma^1)\{1\} \oplus \overline{C}(\Gamma^1) \{-1\}} \stackrel{(f_2, -f_{04}, -f_{14})}{\longrightarrow} \overline{C}(\Gamma^1)\{-1\} \longrightarrow 0,$$

(where $f_{03}, f_{13}$ and $ f_{04}, f_{14}$ are the components of $f_3$ and $f_4$, respectively, under the Second Isomorphism). The later complex decomposes into the direct sum of complexes
\begin{align*}
0 \longrightarrow & \underline{\overline{C}(\Gamma^0)} \longrightarrow 0 \\
0\longrightarrow \overline{C}(\Gamma^1)\{1\} \stackrel{f_{03}}{\longrightarrow} & \underline{\overline{C}(\Gamma^1) \{1\}} \longrightarrow 0 \\
0 \longrightarrow &\underline{\overline{C}(\Gamma^1)\{-1\}} \stackrel{f_{14}}{\longrightarrow}  \overline{C}(\Gamma^1)\{-1\} \longrightarrow 0.\end{align*}

The last two complexes are contractible (this is because the only degree $0$ endomorphisms of $\overline{C}(\Gamma^1)$ are rational multiples of the identity endomorphism, thus $f_{03}$ and $f_{14}$ are isomorphisms). Moreover, $\overline{C}(\Gamma^0) \cong \overline{C}(\Gamma)$ in $hmf_\omega^{fd},$ and we conclude that  $C(D_1)$ and $C(\Gamma)$ are isomorphic in $K_\omega.$

%%%%%%%%%%%%%%%%%%%%%%%%%%%%%%%%%%%%%%%%%%%%%%%%%%%
\textbf{Reidemeister IIb.}
$$D_2=\raisebox{-7pt}{\includegraphics[height=0.25in]{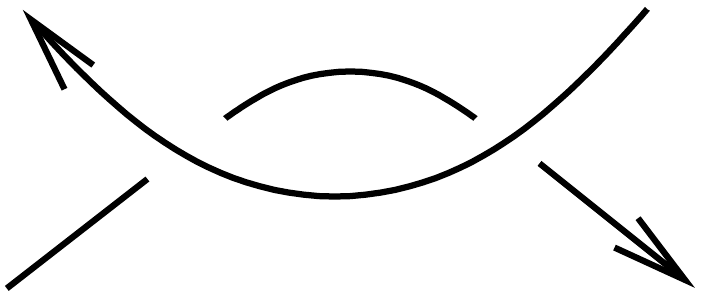}}\qquad
\Gamma'=\raisebox{-7pt}{\includegraphics[height=0.25in]{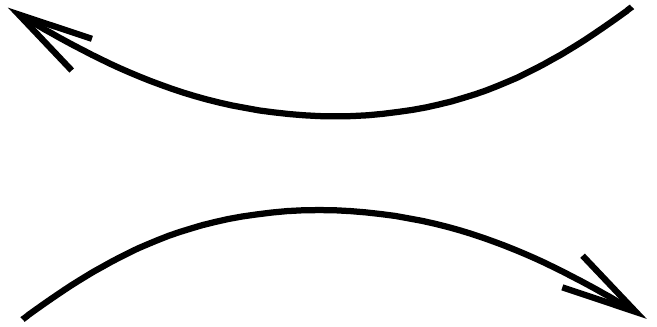}}$$

The complex of matrix factorizations $C(D_2)$ is an element of the category $K_{\omega},$ where $\omega = p(a, h, x_1) - p(a, h, x_2) + p(a, h, x_3) - p(a, h, x_4),$ and has the form
$$ 0\longrightarrow \overline{C}(\Gamma_{00})\{1\} \stackrel{(f_1, f_3)^t}{\longrightarrow} \underline{\overline{C}(\Gamma_{01}) \oplus \overline{C}(\Gamma_{10})} \stackrel{(f_2, - f_4)}{\longrightarrow} \overline{C}(\Gamma_{11})\{-1\} \longrightarrow 0.$$

The resolutions of $D_2$ are given in Figure~\ref{fig:reidIIB}. 
\begin{figure}[ht]
\raisebox{-11pt}{\includegraphics[height=2in]{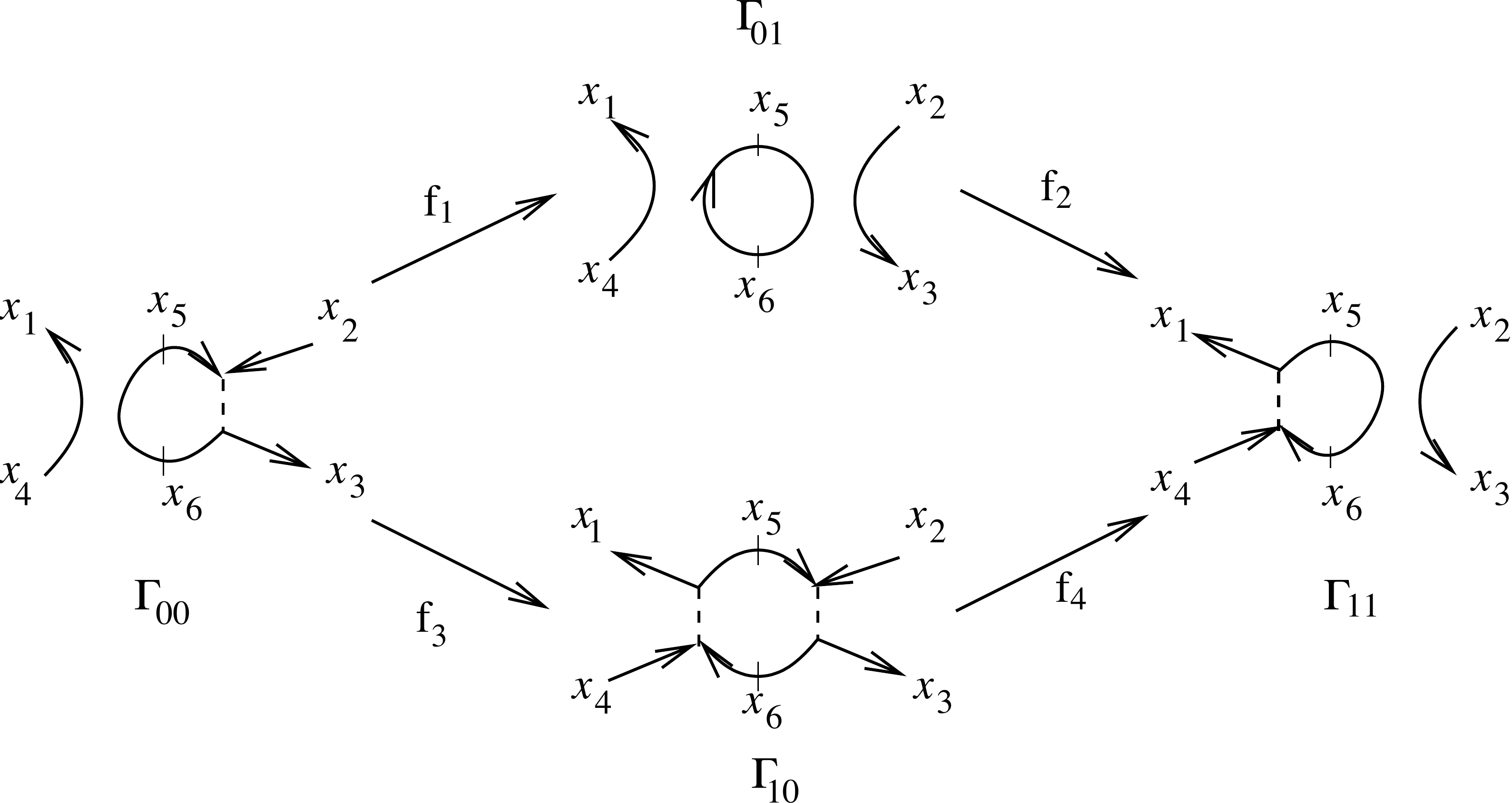}}
\caption{}\label{fig:reidIIB}
\end{figure}

We know that $ \overline{C}(\Gamma_{01}) \cong \overline{C}(\Gamma'') \brak{1} \otimes _{\mathbb{Q}[a,h]} \mathcal{A} \cong \overline{C}(\Gamma'') \brak{1} \{1\} \oplus \overline{C}(\Gamma'')\brak{1} \{-1\},$
$$\overline{C}(\Gamma_{00}) \cong \overline{C}(\Gamma'') \brak{1} \cong \overline{C}(\Gamma_{11})\quad \mbox{and} \quad \overline{C}(\Gamma_{10}) \cong \overline{C}(\Gamma'),$$
where $\Gamma''$ is the diagram in Figure~\ref{fig:Gamma-dprime}. Here we used the First Isomorphism and the Third Isomorphism, and that marking doesn't matter.
\begin{figure}[ht]
\raisebox{-11pt}{\includegraphics[height=.9in]{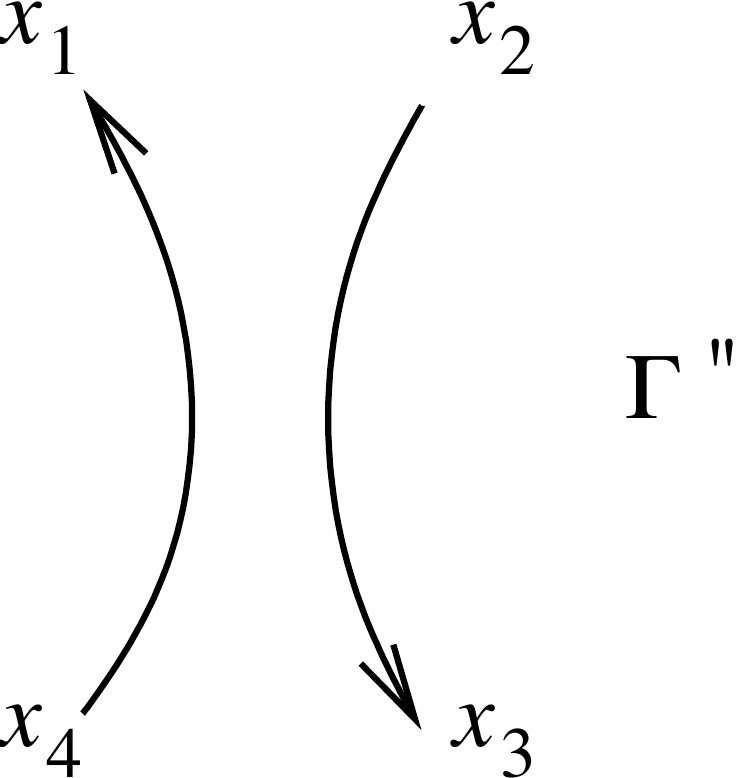}}
\caption{}\label{fig:Gamma-dprime}
\end{figure}

Consequently, $C(D_2)$ is isomorphic to the following complex
$$ 0 \longrightarrow  \overline{C}(\Gamma'') \brak{1}\{1\} \longrightarrow \underline{\overline{C}(\Gamma'') \brak{1} \{1\} \oplus \overline{C}(\Gamma'') \brak{1} \{-1\} \oplus \overline{C}(\Gamma')} \longrightarrow \overline{C}(\Gamma'') \brak{1} \{-1\} \longrightarrow 0,$$
which decomposes into the direct sum of complexes
\begin{align*}
0 \longrightarrow &\underline{\overline{C}(\Gamma')} \longrightarrow 0 \\
0 \longrightarrow \overline{C}(\Gamma'')\brak{1}\{1\} \stackrel{\cong}{\longrightarrow} &\underline{\overline{C}(\Gamma'')\brak{1}\{1\}} \longrightarrow 0 \\
0 \longrightarrow & \underline{\overline{C}(\Gamma'')\brak{1}\{-1\}} \stackrel{\cong}{\longrightarrow} \overline{C}(\Gamma'')\brak{1}\{-1\} \longrightarrow 0.\end{align*}
The last two are contractible, therefore $C(D_2) \cong C(\Gamma')$ in the category $K_{\omega}.$

%%%%%%%%%%%%%%%%%%%%%%%%%%%%%%%%%%%%%%%%%%%%%%%%%%%
\textbf{Reidemeister III.} Given diagrams $D$ and $D'$ below, we show that complexes $C(D)$ and $C(D')$ are isomorphic by showing they are both isomorphic to the same third complex.

$$D \quad \raisebox{-13pt}{\includegraphics[height=.5in]{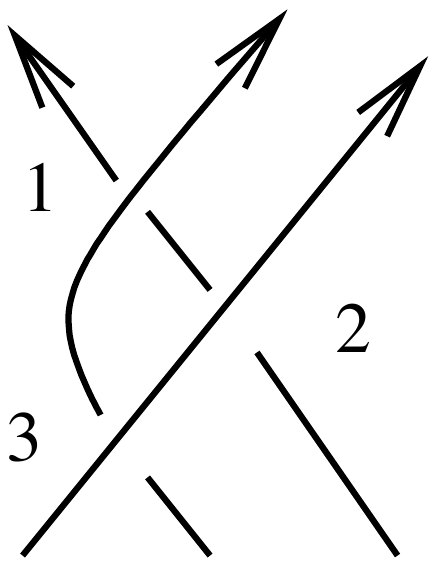}} \hspace{2cm} \raisebox{-13pt}{\includegraphics[height=.5in]{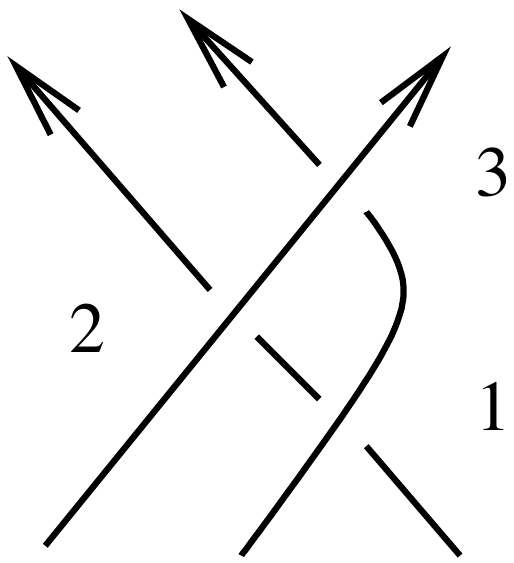}} \quad D'$$

The cube of resolutions corresponding to the diagram $D$ is given in Figure~\ref{fig:reidIII-left}, and that of $D'$ in Figure~\ref{fig:reidIII-right}.
\begin{figure}[ht]
\raisebox{-8pt}{\includegraphics[height=3.5in]{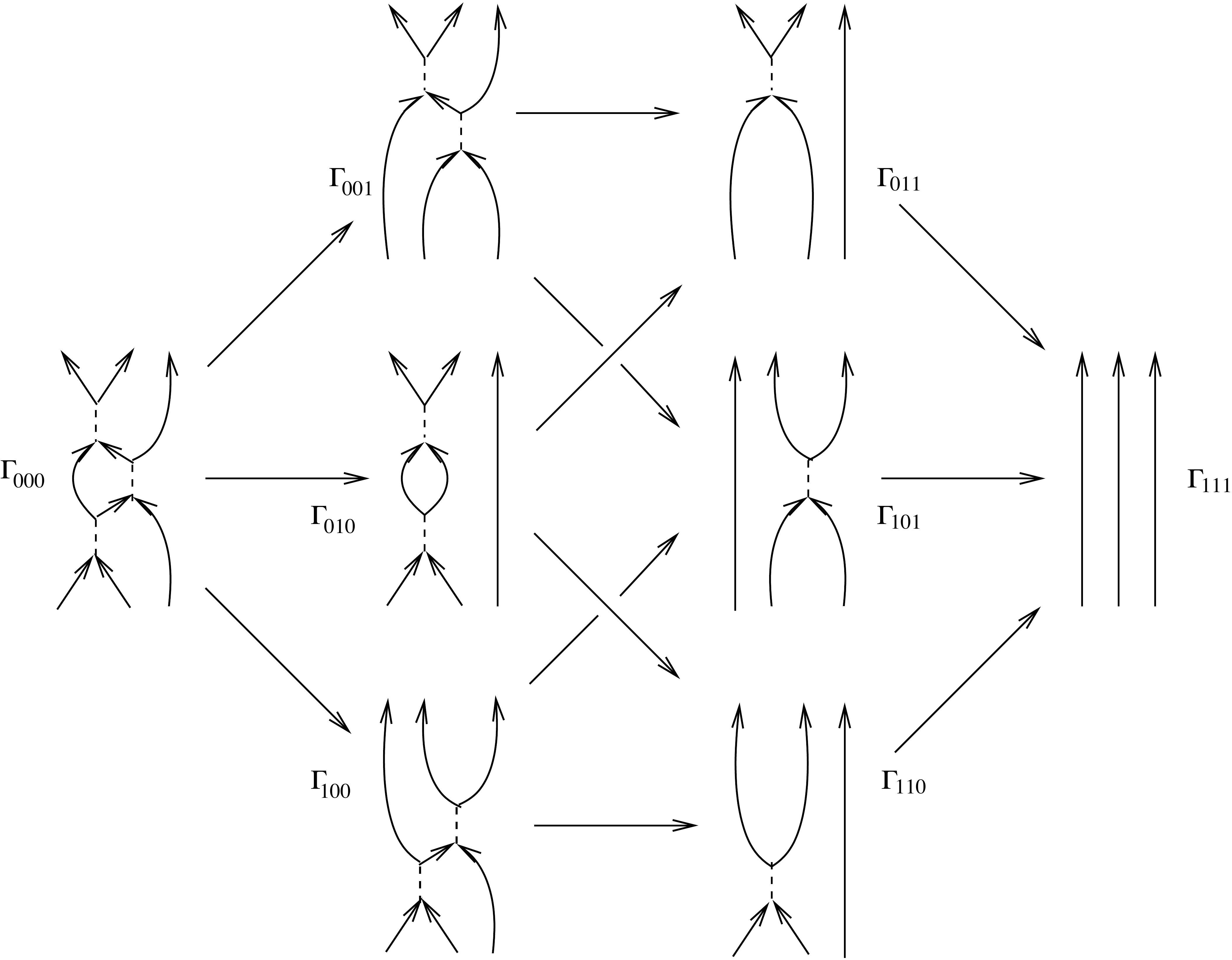}}
\caption{The cube of resolutions of $D$}\label{fig:reidIII-left}
\end{figure}

\begin{figure}[ht]
\raisebox{-8pt}{\includegraphics[height=3.5in]{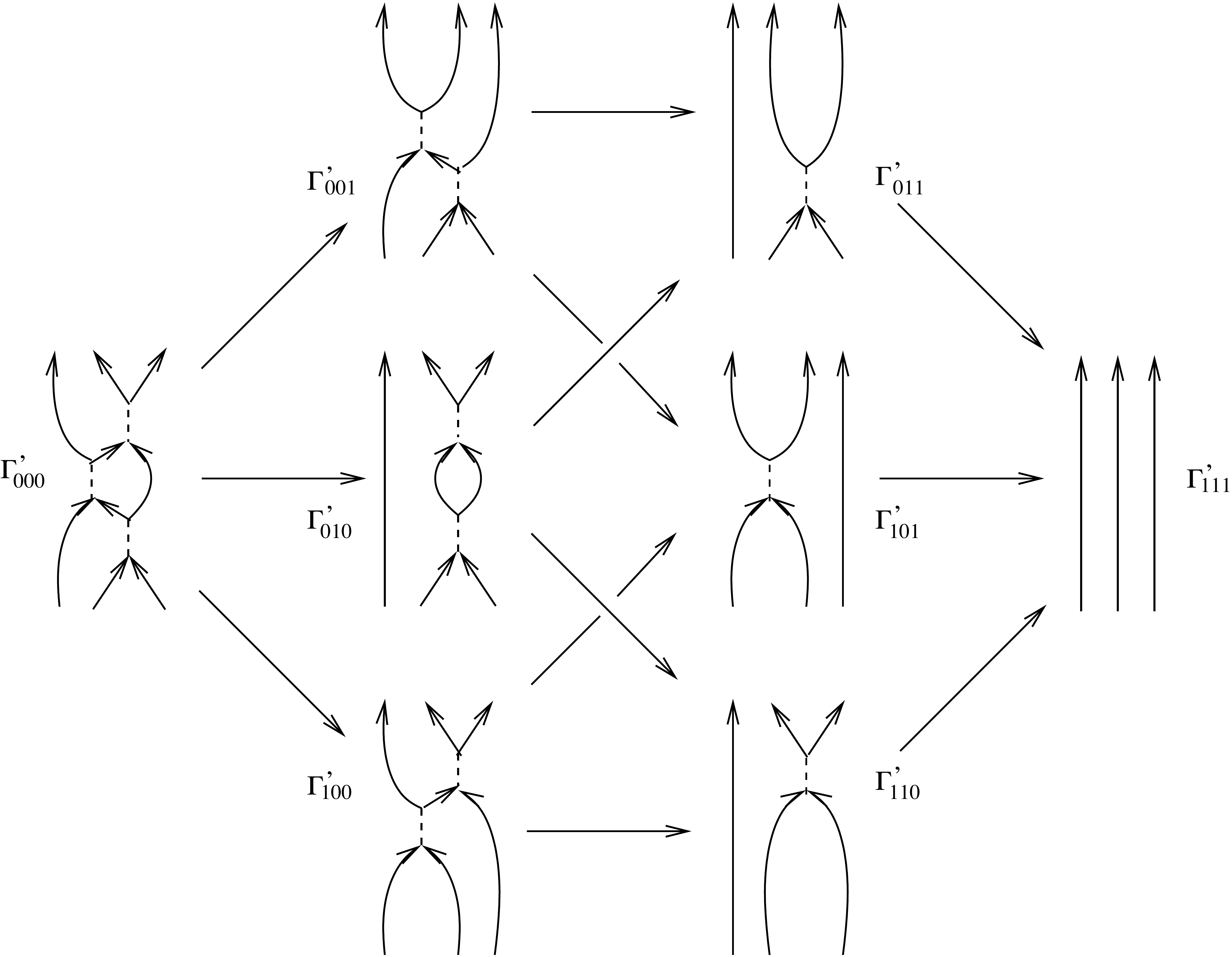}}
\caption{The cube of resolutions of $D'$}\label{fig:reidIII-right}
\end{figure}

The complex $C(D)$ has the form
\[0 \longrightarrow \overline{C}(\Gamma_{000}) \{6\} \stackrel{d^{-3}}{\longrightarrow} \left ( \begin{array}{c} \overline{C}(\Gamma_{001})\{5\} \\ \overline{C}(\Gamma_{010})\{5\} \\ \overline{C}(\Gamma_{100})\{5\} \end{array}  \right ) \stackrel{d^{-2}}{\longrightarrow} \left ( \begin{array}{c} \overline{C}(\Gamma_{011})\{4\} \\ \overline{C}(\Gamma_{101})\{4\} \\ \overline{C}(\Gamma_{110})\{4\} \end{array}  \right) \stackrel{d^{-1}}{\longrightarrow} \underline{\overline{C}(\Gamma_{111})\{3\}} \longrightarrow 0.\]

We know that
\begin{align*}
\overline{C}(\Gamma_{110})  & \cong \overline{C}(\Gamma_{011}) \quad (\text{the two webs are isotopic})\\
 \overline{C}(\Gamma_{010}) &\cong \overline{C}(\Gamma_{011}) \{1\} \oplus \overline{C}(\Gamma_{011}) \{-1\} \quad (\text{by the Second Isomorphism})\\
 \overline{C}(\Gamma_{000}) &\cong \overline{C}(\Gamma_{011}) \quad (\text{by the Fourth Isomorphism}).\end{align*}
 
The complex $C(D')$ has the same form as $C(D).$ Moreover, similar isomorphisms of matrix factorizations as those above hold for resolutions of $D',$ with the remark that $\Gamma_{ijk}$ in $C(D)$ should be replaced by $\Gamma'_{ijk}$ in $C(D').$
 
 $C(D)$ is isomorphic to the following complex
 \[0 \longrightarrow \overline{C}(\Gamma_{011}) \{6\} \stackrel{d^{-3}}{\longrightarrow} \left ( \begin{array}{c} \overline{C}(\Gamma_{001})\{5\} \\ \overline{C}(\Gamma_{011})\{6\} \\ \overline{C}(\Gamma_{011})\{4\}\\ \overline{C}(\Gamma_{100})\{5\} \end{array}  \right ) \stackrel{d^{-2}}{\longrightarrow} \left ( \begin{array}{c} \overline{C}(\Gamma_{011})\{4\} \\ \overline{C}(\Gamma_{101})\{4\} \\ \overline{C}(\Gamma_{011})\{4\} \end{array}  \right) \stackrel{d^{-1}}{\longrightarrow} \underline{\overline{C}(\Gamma_{111})\{3\}} \longrightarrow 0,\]
 but the later decomposes into contractible complexes of the form 
\[0 \longrightarrow \overline{C}(\Gamma_{011}) \{6\} \stackrel{\cong}{\longrightarrow} \overline{C}(\Gamma_{011}) \{6\} \longrightarrow 0 \]
\[0 \longrightarrow \overline{C}(\Gamma_{011}) \{4\} \stackrel{\cong}{\longrightarrow} \overline{C}(\Gamma_{011}) \{4\} \longrightarrow 0 \]
and the complex $\mathcal{C}$
\[\mathcal{C}: \quad 0 \longrightarrow \left ( \begin{array}{c} \overline{C}(\Gamma_{001})\{5\} \\ \overline{C}(\Gamma_{100})\{5\} \end{array}\right )  \longrightarrow \left ( \begin{array}{c} \overline{C}(\Gamma_{011})\{4\} \\ \overline{C}(\Gamma_{101})\{4\} \end{array}\right) \longrightarrow \underline{\overline{C}(\Gamma_{111})\{3\}} \longrightarrow 0.  \]
In other words, complexes $C(D)$ and $\mathcal{C}$ are isomorphic in $K_\omega.$
 
We apply the same argument as in the case of $C(D)$ to conclude that $C(D')$ is isomorphic in $K_\omega$ to the complex $\mathcal{C'}$
\[\mathcal{C'}: \quad 0 \longrightarrow \left ( \begin{array}{c} \overline{C}(\Gamma'_{001})\{5\} \\ \overline{C}(\Gamma'_{100})\{5\} \end{array}\right )  \longrightarrow \left ( \begin{array}{c} \overline{C}(\Gamma'_{101})\{4\} \\ \overline{C}(\Gamma'_{011})\{4\} \end{array}\right) \longrightarrow \underline{\overline{C}(\Gamma'_{111})\{3\}} \longrightarrow 0. \]
Since the web diagram $\Gamma_{ijk}$ from $\mathcal{C}$ and $\Gamma'_{ijk}$ from $\mathcal{C'}$ respectively are isotopic, their matrix factorizations are isomorphic in $hmf_{\omega}.$ In particular $\mathcal{C} \cong \mathcal {C'}$ in $K_\omega,$ and the invariance under the third type of Reidemeister move follows. 
\end{proof}

\begin{corollary}The isomorphism class of the object $C(D)$ in the category $K_\omega$ is an invariant of the tangle $T.$
\end{corollary}

\textbf{The case of links.} When $T$ is a link $L,$ for any resolution $\Gamma$ of $D,$ the corresponding factorization $\overline{C}(\Gamma)$ is a 2-periodic complex of graded $R' = \mathbb{Q}[a,h]$-modules.  The category $K_\omega$ is isomorphic to the category of finite-rank $\mathbb{Z} \oplus \mathbb{Z} \oplus \mathbb{Z}_2$-graded $\mathbb{Q}[a,h]$-modules. For each $\Gamma,$ the homology groups of $\overline{C}(\Gamma)$ are nontrivial only in one degree, thus the cohomology of $C(D)$ is $\mathbb{Z} \oplus \mathbb{Z}$-graded.

We denote by $KR_{a,h}(D)$ the universal Khovanov-Rozansky complex for $n=2.$ It is obtained from $C(D)$ by replacing each matrix factorization $\overline{C}(\Gamma)$ with its homology group $\overline{H}(\Gamma).$ Furthermore, we denote by $HKR_{a,h}$ the cohomology of $KR_{a,h}.$ Note that  $HKR_{a,h}$ is exactly the $\mathbb{Z} \oplus \mathbb{Z}$-graded cohomology of $C(D).$ We have
$$ HKR_{a,h}(D) : = \bigoplus_{i, j \in \mathbb{Z}} H^{i,j}(C(D)). $$ 
It follows from construction that the graded Euler characteristic of $HKR_{a,h}(D)$ is the polynomial $\brak{L}.$ Specifically, the following holds
$$ \brak{L} = \sum_{i,j \in \mathbb{Z}} (-1)^i q^j \rk_{\mathbb{Q}[a,h]} H^{i,j} (D).$$

%%%%%%%%%%%%%%%%%%%%%%%%%%%%%%%%%%%%%%%%%%%%%%%%%%%%Differentials for closed webs
%%%%%%%%%%%%%%%%%%%%%%%%%%%%%%%%%%%%%%%%%%%%%%%%%%%

\section{Understanding the differentials}\label{sec:TQFT}

The ring $\mathcal{A} = \mathbb{Q}[a,h, X]/(X^2 -hX - a)$ is commutative Frobenius with trace map
 $\epsilon \co \mathcal{A} \rightarrow \mathbb{Q}[a, h], \,\,\epsilon(1) = 0, \,\epsilon(X) = 1$.
 Multiplication $m \co \mathcal{A} \otimes \mathcal{A} \rightarrow \mathcal{A}$ and comultiplication $\Delta \co \mathcal{A} \rightarrow \mathcal{A} \otimes \mathcal{A}$ are defined by
$$ \begin{cases}
 m(1 \otimes X) =X, & m(X \otimes 1) = X\\ 
m(1 \otimes 1) =1, & m(X \otimes X) =hX+ a
 \end{cases},\quad
 \begin{cases}
 \Delta(1) = 1 \otimes X + X \otimes 1-h1\otimes 1\\ 
 \Delta(X) = X \otimes X + a 1 \otimes 1.\end{cases} $$

Recall that $\mathcal{A} = \brak {1, X}_{\mathbb{Q}[a, h]}$ is graded with $\deg(1) = -1$ and $\deg(X) = 1.$ The trace $\epsilon$ and unit $\iota$ are maps of degree $-1,$ while multiplication and comultiplication are maps of degree 1.

Let $\Lambda_0^* \co \overline{H}(\raisebox{-8pt}{\includegraphics[height=0.3in]{orcircles.pdf}}) \to \overline{H}(\raisebox{-8pt}{\includegraphics[height=0.3in]{closed-web.pdf}})$ and $\Lambda_1^*\co \overline{H}(\raisebox{-8pt}{\includegraphics[height=0.3in]{closed-web.pdf}}) \to \overline{H}(\raisebox{-8pt}{\includegraphics[height=0.3in]{orcircles.pdf}})$ be the maps induced by $\Lambda_0$ and $\Lambda_1,$ respectively, at the homology level of the corresponding factorizations.

We know that 
\begin{align*}
\overline{C} (\raisebox{-8pt}{\includegraphics[height=0.3in]{closed-web.pdf}}) &\stackrel {\textit{hmf}_\omega}{\cong} \overline{C}(\raisebox{-8pt}{\includegraphics[height=0.3in]{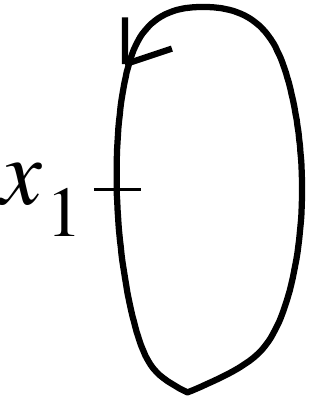}}\,)\brak{1}, \quad \overline{H}(\raisebox{-8pt}{\includegraphics[height=0.3in]{orcircle}}\,) \cong \mathcal{A}, \\
\overline{H}(\raisebox{-8pt}{\includegraphics[height=0.3in]{closed-web.pdf}}) &\cong \mathbb{Q}[a,h,x_1, x_2]/(x_1 + x_2 -h,\, x_1x_2 + a) \{-1 \} \\ & \cong \brak{1, x_1}_{\mathbb{Q}[a,h]}\{-1\}, \\ \overline{H}(\raisebox{-8pt}{\includegraphics[height=0.3in]{orcircles.pdf}}) &\cong \mathbb{Q}[a,h,x_1, x_2]/(x_1^2 -hx_1 - a, \, x_2^2 -hx_2 - a) \{-2\} \\
&\cong \brak{1, x_1, h-x_2, x_1(h-x_2)}_{\mathbb{Q}[a,h]} \{-2\}.\end{align*}
 
There are $\mathbb{Q}[a, h]$-module isomorphisms:
\begin{align} \label{isomorphism f}
 \overline{H}(\raisebox{-8pt}{\includegraphics[height=0.3in]{closed-web.pdf}}) \stackrel{f}{\longrightarrow} \mathcal{A}, \quad & f \co \begin{cases} 1 \to 1\\ x_1 \to X  \end{cases} \\ \label{isomorphism g}
 \overline{H}(\raisebox{-8pt}{\includegraphics[height=0.3in]{orcircles.pdf}}) \stackrel{g}{\longrightarrow} \mathcal{A} \otimes \mathcal{A}, \quad & g\co \begin{cases}  1\to 1\otimes 1,\,\, x_1 \to X \otimes 1 \\ h-x_2 \to 1 \otimes X, \,\, x_1(h-x_2) \to X\otimes X.\end{cases} \end{align}
 
 \begin{lemma}
 $\Lambda_0^* = m$ and $\Lambda_1^* = \Delta.$
 \end{lemma}
 \begin{proof}
 We know that $\Lambda_0/(x_1 = x_4, x_2 = x_3) = \id \otimes \,\psi'_{x_1-x_2}$ and $\Lambda_1/(x_1 = x_4, x_2 = x_3)  = \id \otimes \,\psi_{x_1-x_2},$ that factorizations $\overline{C}(\Gamma^0)/(x_1 = x_4, x_2 = x_3) = \overline{C}(\raisebox{-8pt}{\includegraphics[height=0.3in]{orcircles.pdf}})$ and $\overline{C}(\Gamma^1)/(x_1 = x_4, x_2 = x_3) = \overline{C}(\raisebox{-8pt}{\includegraphics[height=0.3in]{closed-web.pdf}})$ have nontrivial homology only in degree 0. 
 
 Applying the definition of the morphism $\psi'_{x_1-x_2}$ we have 
 \begin{align*} \Lambda^*_0(x_1) &= x_1, &\Lambda^*_0(h-x_2) &= h-x_2 \\
\Lambda^*_0(1) &= 1, &\Lambda^*_0 (x_1(h-x_2)) &= x_1(h-x_2).\end{align*} But $h-x_2 = x_1$ and $x_1(h-x_2) = hx_1 + a$ in $\overline{H}(\raisebox{-8pt}{\includegraphics[height=0.3in]{closed-web.pdf}}),$ thus using the isomorphisms \ref{isomorphism f} and \ref{isomorphism g} it follows that $\Lambda_0^* = m.$ Similarly, using the definition of $\psi_{x_1-x_2}$ we obtain
 \begin{align*} \Lambda^*_1(1) &= x_1-x_2 \\ \Lambda^*_1(x_1) &= (x_1-x_2)x_1.\end{align*} 
 But $(x_1-x_2)x_1 =  hx_1 +a -x_1x_2 = a + x_1(h-x_2)$ in $\overline{H}(\raisebox{-8pt}{\includegraphics[height=0.3in]{orcircles.pdf}}),$ and $x_1-x_2 = x_1 +(h -x_2) - h.$ By using again the isomorphisms \ref{isomorphism f} and \ref{isomorphism g} we obtain $\Lambda^*_1 = \Delta.$
\end{proof}

It is well known that the commutative Frobenius algebra $\mathcal{A}$ gives rise to a TQFT functor---denoted here by $\mathsf{F}$---from the category $OrCob$ of oriented $(1+1)$--dimensional cobordisms to the category $\mathbb{Q}[a,h]$-Mod of graded $\mathbb{Q}[a, h]$-modules and module homomorphisms. The functor assigns the ground ring $\mathbb{Q}[a, h]$ to the empty 1-manifold, and $\mathcal{A}^{\otimes k}$ to the disjoint union of oriented $k$ circles. On the generating morphisms of $OrCob$, the functor associates the structure maps of the algebra $\mathcal{A}.$ In  particular, $\mathsf{F}(\raisebox{-3pt}{\includegraphics[height=0.15in]{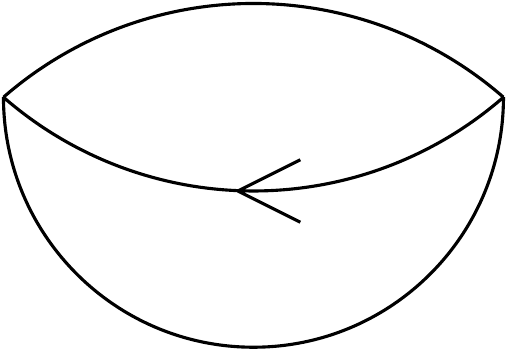}}) = \iota,\, \mathsf{F}(\raisebox{-3pt}{\includegraphics[height=0.15in]{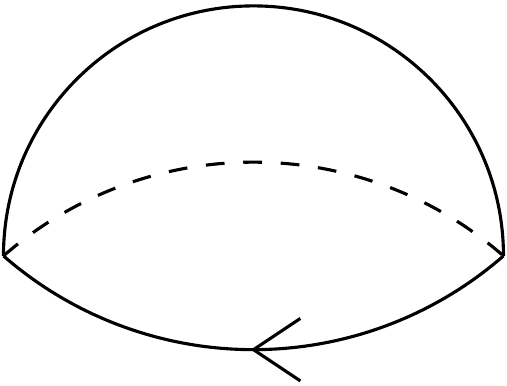}}) = \epsilon,\, \mathsf{F}(\raisebox{-3pt}{\includegraphics[height=0.16in]{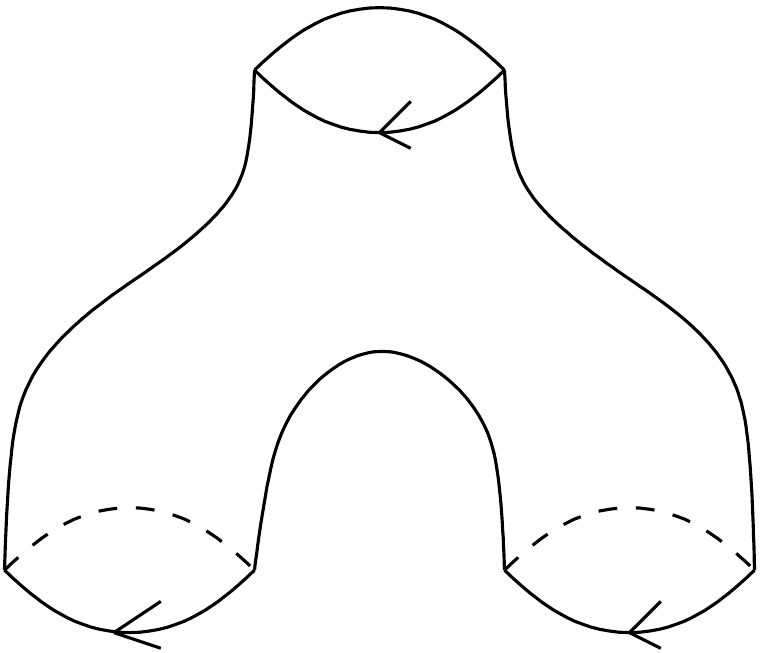}}) = m$ and\, $\mathsf{F}(\raisebox{-3pt}{\includegraphics[height=0.16in]{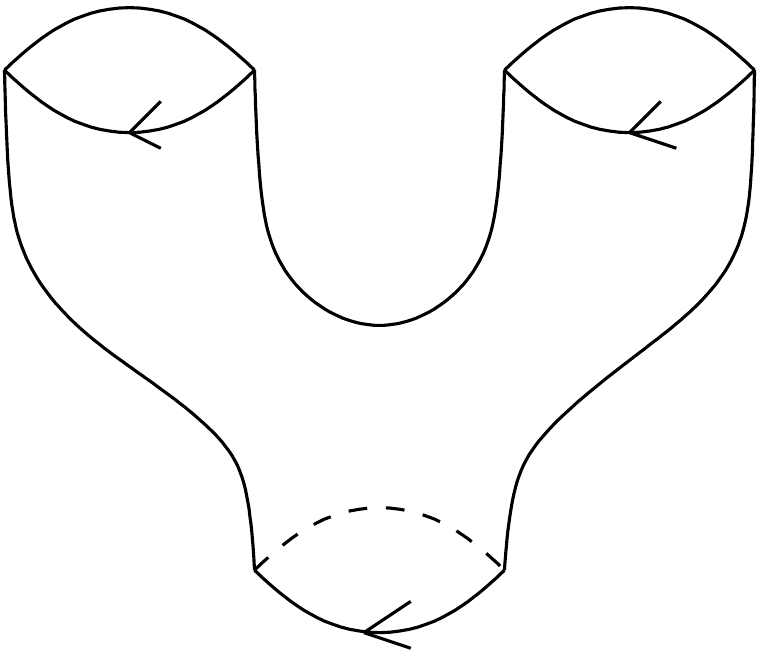}}) = \Delta$ (note that we read cobordisms from bottom to top). 

We remark that the objects in $OrCob$ are oriented in such a way that a cobordism between two objects is ``properly'' oriented. For that, one needs to orient each nesting (concentric)  set of circles so that  orientations alternate, giving to the outermost circle the same orientation, for all nesting sets of circles. 

Given a cobordism $S \in OrCob,$ the homomorphism $\mathsf{F}(S)$ has degree given by the formula $\deg(S) = -\chi(S),$ where $\chi$ is the Euler charecteristic for $S.$ 

Consider a link diagram $D$ with its hypercube of resolutions and associated complex of matrix factorizations $C(D).$ Each resolution $\Gamma$ of $D$ is a collection of closed webs with an even number of vertices and oriented loops. The homology group $\overline{H}(\Gamma)$ of the associated factorization $\overline{C}(\Gamma)$ satisfies $\overline{H}(\Gamma) \cong \mathcal{A}^k,$ where $k$ is the number of connected components of $\Gamma.$ 

We can replace each resolution of $D$ by a simplified one that has fewer number of vertices, via the First, Third and Fourth Isomorphisms. After this operation we are left with a complex of factorizations isomorphic in $K_\omega$ to $C(D),$ in which each resolution is a disjoint union of basic closed webs (with exactly two vertices) and oriented loops. Finally, applying the First Isomorphism to all basic closed webs, the latter complex is isomorphic in $K_\omega$ to a complex of factorizations whose underlying geometric objects are column vectors of nested oriented loops. Each nested set of loops is oriented in such a way that the outermost loop is oriented clockwise, by convention, and as we go inside of the nesting set of loops orientations alternate. In particular, the latter formal complex, call it $\mathcal{C},$ is an object in the category of complexes over $OrCob.$ Applying the functor $\mathsf{F}$ to $\mathcal{C},$ we obtain an ordinary complex $\mathsf{F}(\mathcal{C})$ whose objects are $\mathbb{Q}[a,h]$-modules and whose differentials are $\mathbb{Q}[a,h]$-module homomorphisms. Moreover, $\mathsf{F}(\mathcal{C})$ is homotopy equivalent to $KR_{a,h}(D).$ 
Consequently, the following corollary is implied by the results of this section.

\begin{corollary}
The functor $\overline{H}$ behaves in the same manner as $\mathsf{F}$ does. 
\end{corollary}

Recalling the author's construction and results in~\cite{CC2}, the previous Corollary implies that the underlying link cohomology and that introduced in~\cite{CC2} are isomorphic, after tensoring them with appropriate rings. In particular, the functor $\overline{H}$ is the same as the tautological functor $\mathcal{F}$ in~\cite{CC2} (at least when we restrict to the case of links). Moreover, this implies that the theory constructed here is functorial under link cobordisms, relative to boundaries.

\begin{corollary}\label{cor:isomorphism}
For each oriented link $L$ there is an isomorphism 
$$HKR_{a,h}(L) \otimes_{\mathbb{Q}[a,h]} \mathbb{Z}[i] \cong \mathcal{H}(L) \otimes_{\mathbb{Z}[i][a,h]} \mathbb{Q}$$
where $\mathcal{H}$ is the cohomology link theory explored in~\cite{CC2}. \end{corollary}
 
\begin{corollary}
Given a link cobordism $C \co L_1 \to L_2,$ there is a well-defined induced map $HKR_{a,h}(C) \co HKR_{a,h}(L_1) \to HKR_{a,h}(L_2)$ between the associated cohomology groups.
\end{corollary}

%%%%%%%%%%%%%%%%%%%%%%%%%%%%%%%%%%%%%%%%%%%%%%%%%%%%the algebra 

\section{The algebra $\overline{R}(\Gamma)$}

Let $\Gamma$ be a resolution of a link $L$ and denote by $e(\Gamma)$ the set of all edges in $\Gamma.$ We introduce an algebra $\overline{R}(\Gamma)$ and exhibit the $\overline{R}(\Gamma)$ module structure of $\overline{H}(\Gamma).$

\begin{definition}
We define $\overline{R}(\Gamma) : = R/(\textbf{a}, \textbf{b}),$ where $R = \mathbb{Q}[a,h][X_i\vert i \in e(\Gamma)]$ and $\textbf{a}, \textbf{b}$ are the polynomials that are used to define the factorization $\overline{C}(\Gamma).$
\end{definition}

$\overline{H}(\Gamma)$ is an $\overline{R}(\Gamma)$ module, since multiplication by any polynomial in $(\textbf{a},\textbf{b})$ induces a null-homotopic endomorphisms of $\overline{C}(\Gamma).$

\begin{proposition}\label{relations}
The algebra $\overline{R}(\Gamma)$ is spanned by generators $X_i,$ where $i\in e(\Gamma),$ subject to the following relations:
\begin{enumerate}
\item For every $X_i$ we have $X_i^2 = h X_i + a.$
\item For every singular resolution (of the form $\Gamma^1$) in $\Gamma,$ the generators $X_1, X_2, X_3, X_4$ satisfy $X_1 + X_2 = h = X_3 + X_4,$ and $X_1 X_2 = -a = X_3 X_4.$
\end{enumerate} 
\end{proposition}
\begin{proof}
For each closed circle---with one mark $i$---in $\Gamma,$ the expression $\overline{\pi}_{ii} = 3(x_i^2 - hx_i -a)$ lives in $(\textbf{a},\textbf{b}).$ For any other $X_i,$ we use that the multiplication by $\partial_i \omega= 3(x_i^2 -hx_i -a)$ induces a null-homotopic endomorphism of $\overline{C}(\Gamma).$ Thus $X_i^2 = h X_i + a$ holds for any $X_i.$ To prove the second statement we recall that for each singular resolution with arcs labeled as those of $\Gamma^1,$ the polynomials $x_1 + x_2 -x_3 -x_4$ and  $x_1x_2 -x_3x_4$ are in $(\textbf{a},\textbf{b}),$ thus $X_1 + X_2 = X_3+X_4$ and $X_1X_2 = X_3X_4$ in  $\overline{R}(\Gamma),$ near each singular resolution of $\Gamma.$ Moreover, since $\overline{u}_2 = 3(x_3 +x_4) -3h$ and $\overline{u}_1 = (x_1 + x_2)^2 + (x_1 + x_2)(x_3 + x_4) + (x_3 + x_4)^2 -3x_1x_2 -\frac{3}{2} h( x_1+x_2+x_3 + x_4) -3a$ are also in $(\textbf{a},\textbf{b}),$ we have $X_1 + X_2 = h = X_3 + X_4,$ and $X_1 X_2 = -a = X_3 X_4.$
\end{proof}

\begin{remark}
Relations given in Proposition~\ref{relations} have a geometric interpretation via a TQFT with dots, where a dot stands for the multiplication by $X$ endomorphism of the algebra $\mathcal{A}.$ Specifically, the relation $X_i^2 = h X_i + a$ translates into the following skein relation: a disk decorated by two dots equals a disk decorated by one dot times $h$ plus a disk times $a$ (this one is the relation (2D) in~\cite{CC2}). Given two edges in $\Gamma$ with labels $i$ and $j$ and which share a vertex, the relation $X_i + X_j = h$ gives a rule for exchanging dots between two neighboring facets of a foam, while relation $X_i X_j = -a$ means that if each of the two neighboring facets have a single dot, we can ``erase'' both dots and multiply the corresponding foam by $-a$ (see relations (ED) in~\cite[Figure 7]{CC2}).
\end{remark}

It was proved in\cite{CC2} that if one lets $a$ and $h$ be complex numbers and considers the polynomial $f(X) = X^2 -hX -a \in \mathbb{C}[X],$ the isomorphism class of the complex $\mf{sl}(2)$-foam cohomology $\mathcal{H}(L, \mathbb{C})$ is determined by the number of distinct roots of $f[X].$ Corollary~\ref{cor:isomorphism} implies that this also holds for the complex matrix factorization cohomology $HKR_{a,h}(L, \mathbb{C}).$ The results are as follows.

If $f(X) = (X-\alpha)^2,$ for some $\alpha \in \mathbb{C},$ there is an isomorphism between $HKR_{a,h}(L, \mathbb{C})$ and the original $\mf{sl}(2)$ Khovanov-Rozansky cohomology over $\mathbb{C}.$

If $f(X) = (X -\alpha)(X-\beta),$ for some $\alpha, \beta \in \mathbb{C},$ for each resolution $\Gamma$ of a link $L,$ the cohomology $\overline{H}(\Gamma)$ is a free module of rank one over the complex algebra $\overline{R}(\Gamma).$ 

\begin{proposition}
For any $n$-component link $L,$ the dimension of $HKR_{a,h}(L, \mathbb{C})$ equals $2^n,$ and to each map $\phi \co \{ \text{components of L} \} \to S = \{ \alpha, \beta \}$ there exists a non-zero element $h_{\phi} \in HKR_{a,h}(L, \mathbb{C})$ which lies in the cohomological degree
$$ -2 \sum_{\substack {(u_1,u_2)\in S \times S \\ u_1 \neq u_2}} lk(\phi^{-1}(u_1), \phi^{-1}(u_2)),$$
and all $h_{\phi}$ generate $HKR_{a,h}(L, \mathbb{C}).$
\end{proposition}


\begin{thebibliography}{999}
\bibitem{BN} D. Bar-Natan, {\em Khovanov's homology for tangles and cobordisms}, 
 Geom.Topol. \textbf{9} (2005), 1443-1499

\bibitem{CC1} C. Caprau, {\em sl(2) tangle homology with a parameter and singular cobordisms}, to appear in Algebr. Geom. Topology 

\bibitem{CC2} C. Caprau, {\em The universal sl(2) cohomology via webs and foams}, arXiv:math.GT/0802.2848

\bibitem{G} B. Gornik, {\em Note on Khovanov link cohomology}, arXiv:math.QA/0402266

\bibitem{Kh1} M. Khovanov, {\em A categorification of the Jones polynomial}, Duke Math.J. \textbf{101} (2000) no. 3, 359-426

\bibitem{Kh2} M. Khovanov, {\em $\mf{sl}(3)$ link homology}, Algebr. Geom. Topol. \textbf{4} (2004), 1045-1081.

\bibitem{KhR1} M. Khovanov, L.Rozansky, {\em Matrix factorizations and link homology}, to appear in Fundamenta Mathematicae, arXiv:math.QA/0401268

\bibitem{KhR2} M. Khovanov, L.Rozansky, {\em Matrix factorizations and link homology II}, to appear in Geometry and Topology, arXiv:math.QA/0505056

\bibitem{MV1} M. Mackaay, P. Vaz, {\em The universal $\mf{sl}(3)$-link homology}, Algebr. Geom. Topol. \textbf{7} (2007) 1135-1169

\bibitem{MV2} M. Mackaay, P. Vaz, {\em The foam and the matrix factorization $\mf{sl}(3)$ link homologies are equivalent}, Algebr. Geom. Topol. \textbf{8} (2008) 309-342

\bibitem{MOY} H. Murakami, T. Ohtsuki and S. Yamada, {\em HOMFLY polynomial via an invariant of colored plane graphs}, Enseign. Math. (2) 44 (1998), no.3-4, 325-360 

\bibitem{Ras2} J.A. Rasmussen, {\em Some differentials on Khovanov-Rozansky homology}, arXiv:math.GT/0607544

\bibitem{W} H. Wu, {\em On the quantum filtration of the Khovanov-Rozansky cohomology}, arXiv:math.GT/0612406
\end{thebibliography}
\end{document}